\documentclass{article}
\usepackage[utf8]{inputenc}
\usepackage[utf8]{inputenc}
\usepackage{latexsym}
\usepackage{amsfonts}
\usepackage{amsthm}
\usepackage{amsmath}
\usepackage{graphicx}

\newtheorem{theorem}{Theorem}[section]
\newtheorem{lemma}[theorem]{Lemma}
\newtheorem{definition}[theorem]{Definition}
\newtheorem{corollary}[theorem]{Corollary}
\newtheorem{conjecture}[theorem]{Conjecture}

\theoremstyle{remark}
\newtheorem{remark}[theorem]{Remark}
\newtheorem{example}[theorem]{Example}

\def\Z{\mathbb{Z}}
\def\Q{\mathbb{Q}}
\def\N{\mathbb{N}}
\def\R{\mathbb{R}}
\def\T{{\cal{T}}}
\def\gpB{{\cal{B}}}
\def\gpC{{\cal{C}}}
\def\RT{{\cal{RT}}}
\def\bfw{{\bf w}}

\def\bfB{{\bf B}}
\def\bfC{{\bf C}}
\def\bfU{{\bf U}}
\def\bfV{{\bf V}}
\def\bfn{{\bf n}}
\def\backn{{\overset{\leftarrow}{\bfn}}}
\def\backC{{\overset{\longleftarrow}{\bf C}}}

\def\s{{\cal S}}

\title{Jones rational coincidences}
\author{Ruth Lawrence\footnote{Einstein Institute of Mathematics, Hebrew University of Jerusalem}$\ $ and Ori Rosenstein$^*$}
\begin{document}

\maketitle
\centerline{\it In memory of Vaughan Jones} 
\medskip

\noindent \textbf{Abstract:}
We investigate coincidences of the (one-variable) Jones polynomial amongst rational knots, what we call `Jones rational coincidences'. We provide moves on the continued fraction expansion of the associated rational which we prove do not change the Jones polynomial and conjecture (based on experimental evidence from all rational knots with determinant $<900$) that these moves are sufficient to generate all Jones rational coincidences. These coincidences are generically not mutants, as is verified by checking the HOMFLYPT polynomial. In the process we give a new formula for the Jones polynomial of a rational knot based on a continued fraction expansion of the associated rational, which has significantly fewer terms than other formulae  known to us. The paper is based on the second author's Ph.D.~thesis and gives an essentially self-contained account.



\section{Introduction}
The (one-variable) Jones polynomial was defined in \cite{Jones} as an ambient isotopy invariant of links in $S^3$. Although it has subsequently been found to be connected with many interesting other mathematical structures (Temperley-Leib algebra, quantum group $U_qsl_2$, topological quantum field theory etc.), many basic properties remain unknown. For example, it is still unknown whether Jones' conjecture holds, that $V_K$ distinguishes the unknot from other knots, that is $V_K\equiv1$ for a knot $K$ only if $K$ is the unknot; see \cite{TS} for numerical evidence up to 22 crossings. It is known that Conway mutation \cite{Conway} and various generalisations generate knots with identical Jones polynomial, see \cite{P} and references therein, as well as \cite{Jones92}. Kanenobu \cite{Kanenobu} constructed a very special infinite family of 3-bridge knots $K_{p,q}$ ($p,q\in\Z$) whose Jones polynomial depends only on $p+q$; that is, an infinite family of infinite families each sharing the same Jones polynomial. Further examples are in \cite{Kanenobu86}. It is still unclear whether coincidences should be considered more the rule than the exception.

In this paper we concentrate on coincidences amongst Jones' polynomials of rational knots. Since the determinant of a knot is a special value of the Jones polynomial and there are only a finite number of rational knots with given determinant, the number of coincidences is limited and can be checked by direct computation.

In sections 2,3,4 we give an elementary and self-contained account of the evaluation of the Jones polynomial of rational knots as a matrix product (Theorems \ref{Jonesrationalformula} and \ref{Jonesevenformula}) via the Kauffman bracket vector of $(2,2)$-tangles. Similar matrix formulae were used in \cite{Duzhin} and \cite{KanenobuSumi} for computer computations of the Jones polynomial and other polynomial invariants of rational knots.

By a careful analysis of cancellations of terms in the evaluation of the matrix product, we obtain in section 5 a seemingly new combinatorial formula (Theorem \ref{jonesqnewformula}) for the Jones polynomial of rational knots in terms of a continued fraction expansion, with significantly fewer terms than other formulae we found in the literature, such as in \cite{Duzhin}.

By analyzing properties of products of the particular family of matrices appearing, we prove in section 6 the invariance of the Jones polynomial of rational knots under a family of moves, Templates I, II (Theorems \ref{firsttemplate}, \ref{secondtemplate}). A third move, a generalization of Template I, is introduced in section 7 (Theorem \ref{pivotingpairstemplate}) using the notion of pivoting pairs defined there. On the way a new structural property of Jones polynomials of rational knots is found (Corollary \ref{Jonesrationalform}). In \cite{R} it was found that out of 80,317 rational knots with determinant $<900$, there are 223 Jones rational coincidences, three of which are triplets; these are listed in \S8. All of these coincidences were accounted for via the moves proved in sections 6,7 and we conjecture that they are indeed generating moves for Jones rational coincidences.

\section{Kauffman bracket}
Let $A$ be an indeterminate and $R=\Z[A,A^{-1}]$ be the ring of polynomials in $A,A^{-1}$ with integer coefficients. Set $d=-A^2-A^{-2}\in{}R$, $u=-iA^2$ and $t=A^{-4}$.

The usual Kauffman bracket algorithm \cite{Kauffman} to calculate the Jones polynomial of an oriented link $L$ in $S^3$ works as follows. Choose a link diagram $D$ of $L$   and evaluate it (as an unoriented diagram) by replacing every crossing by a combination of diagrams  in which the crossing has been smoothed in the two possible ways
\begin{equation}\label{eq:skeinrelation}
\includegraphics[width=.5\textwidth]{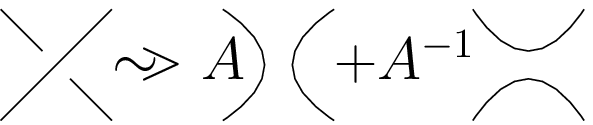}%\eqno{(8)}
\end{equation}
A diagram with $n$ crossings is thus replaced by a combination of $2^n$ diagrams, each of which has no crossings, that is, all resulting diagrams are isotopic to a disjoint collection of loops; the coefficients are powers of $A$. Evaluating a diagram which is a disjoint collection of $k$ loops to $d^{k-1}$, the combined evaluation of $D$ is its Kauffman bracket, written $\langle{}D\rangle{}\in{}R\equiv\Z[A,A^{-1}]$.

Equivalently, one can consider the appropriate skein modules. Namely in the $R$-module generated by all link diagrams, one imposes the skein relation identifying the two sides of (1) considered as a local move $D\sim{}AD_\infty+A^{-1}D_0$ and identifies any disjoint union with an unknot, $D\coprod{}U$ with $d\cdot{}D$ where $U$ is the diagram of the unknot consisting of a single loop. The quotient is one-dimensional, any link diagram will be equivalent to some multiple of $U$, this multiple being the Kauffman bracket. This is a regular isotopy invariant (invariant under Reidemeister moves II, III) and its renormalisation
\begin{equation}\label{eq:Jonespoly}
V_L=(-A^3)^{-w(D)}\langle{}D\rangle
\end{equation}
by the writhe $w(D)$ is an ambient isotopy invariant, the (one-variable) Jones polynomial $V_L$ of $L$ \cite{Jones}. For knots, $V_L$ is a polynomial in $t,t^{-1}$ where $t=A^{-4}$.
Recall that the writhe of a diagram is defined as the sum over all crossings of the the sign of the crossing (determined by the relative orientations of the segments on $D$ involved in the crossing); we use the sign convention shown below.
\[
\includegraphics[width=.2\textwidth]{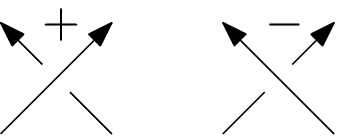}
\]

\section{Calculus of $(2,2)$-tangles}
Let $\T$ denote the monoid of framed unoriented $(2,2)$-tangles considered up to ambient isotopy equivalence, under vertical composition, or equivalently (using blackboard framing) of diagrams of $(2,2)$-tangles up to Reidemeister moves II and III. The identity is the tangle $T_\infty$ (see below).
 Then $R\T$ defines an associative unital algebra over $R$, whose elements are finite linear combinations of $(2,2)$-tangles with coefficients in $R$. Call a tangle {\it loopless} if every component contains two  boundary points, that is, it is an embedding of the disjoint union of two intervals in $\R^2\times[0,1]$.

Tangles can be manipulated in various ways to produce new tangles. Elements of $\T$ can be rotated in the plane of the diagram ($T$ becomes $\vec{T}$ under clockwise rotation through $\frac\pi2$) as well as combined side-by-side ($T$ and $T'$ combine to give $T+T'$). They can also be closed horizontally or vertically to obtain (unoriented) link diagrams; we denote these by $T^N$ (numerator closure) and $T^D$ (denominator closure) respectively.
\[
\includegraphics[width=.9\textwidth]{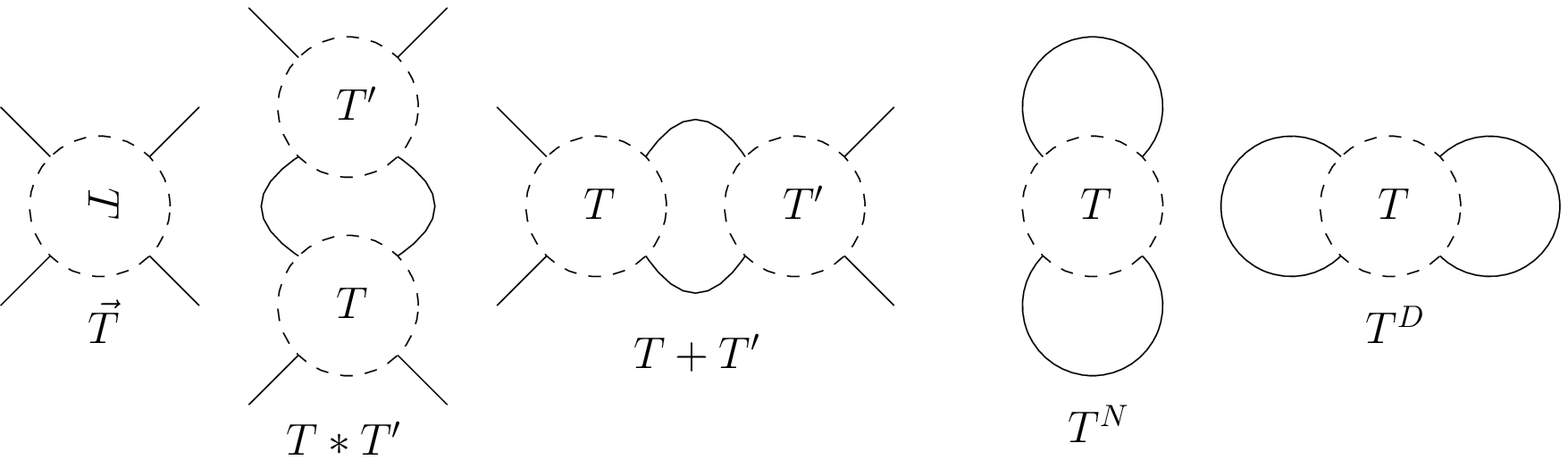}
\]
These operations are related, for example $\overset{\longrightarrow}{(T*T')}=\vec{T}+\vec{T'}$ while $(\vec{T})^N=T^D$. The operation $+$ also makes $\T$ into a monoid (a different structure from $*$); the identity here is $T_0$,
\[
\includegraphics[width=.3\textwidth]{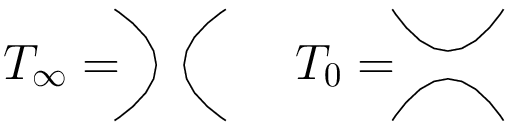}
\]
{\bf Tangle type.} Any $(2,2)$-tangle $T$ provides a subdivision of the four boundary points into two pairs, via the connections it provides (ignoring which paths pass over or under). There are three possible subdivisions of four points into two pairs and we define the {\it tangle type} of $T$ accordingly, to be 0, 1 or $\infty$ as in the diagram below. We denote tangle type by $\iota(T)\in\{0,1,\infty\}$.
\[
\includegraphics[width=.3\textwidth]{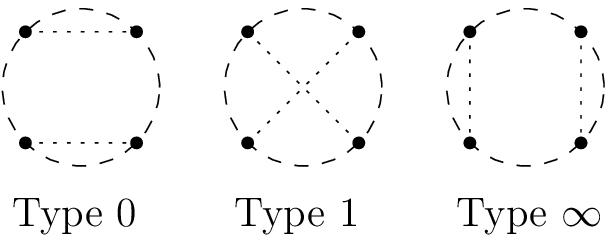}
\]
Under rotation by $\frac{\pi}2$, types 0 and $\infty$ interchange while type 1 stays fixed, so we write $\iota(\vec{T})=\iota(T)^{-1}$. Under the tangle operation $+$, the tangle types add, where addition is modulo 2 enhanced by the rule $x+\infty=\infty+x=\infty$ for all $x$,
\[\iota(T+T')=\iota(T)+\iota(T')\]
The numerator closure $T^N$ of a loopless tangle $T$  will be a knot so long as $\iota(T)\not=0$.

\medskip\noindent{\bf Writhe vector.} The definition of writhe uses an orientation on the link diagram. In the case of a knot, the writhe is independent of the choice of orientation while this is not true for links. For any loopless $(2,2)$-tangle $T$, the numerator closure $T^N$ is a knot when $\iota(T)\not=0$ and is a 2-component link otherwise, while the denominator closure $T^D$ is a knot when $\iota(T)\not=\infty$ and is a 2-component link otherwise. Define $T^N$, $T^D$ as oriented link diagrams by placing orientations as indicated, according to the tangle type of $T$.
\[
\includegraphics[width=.6\textwidth]{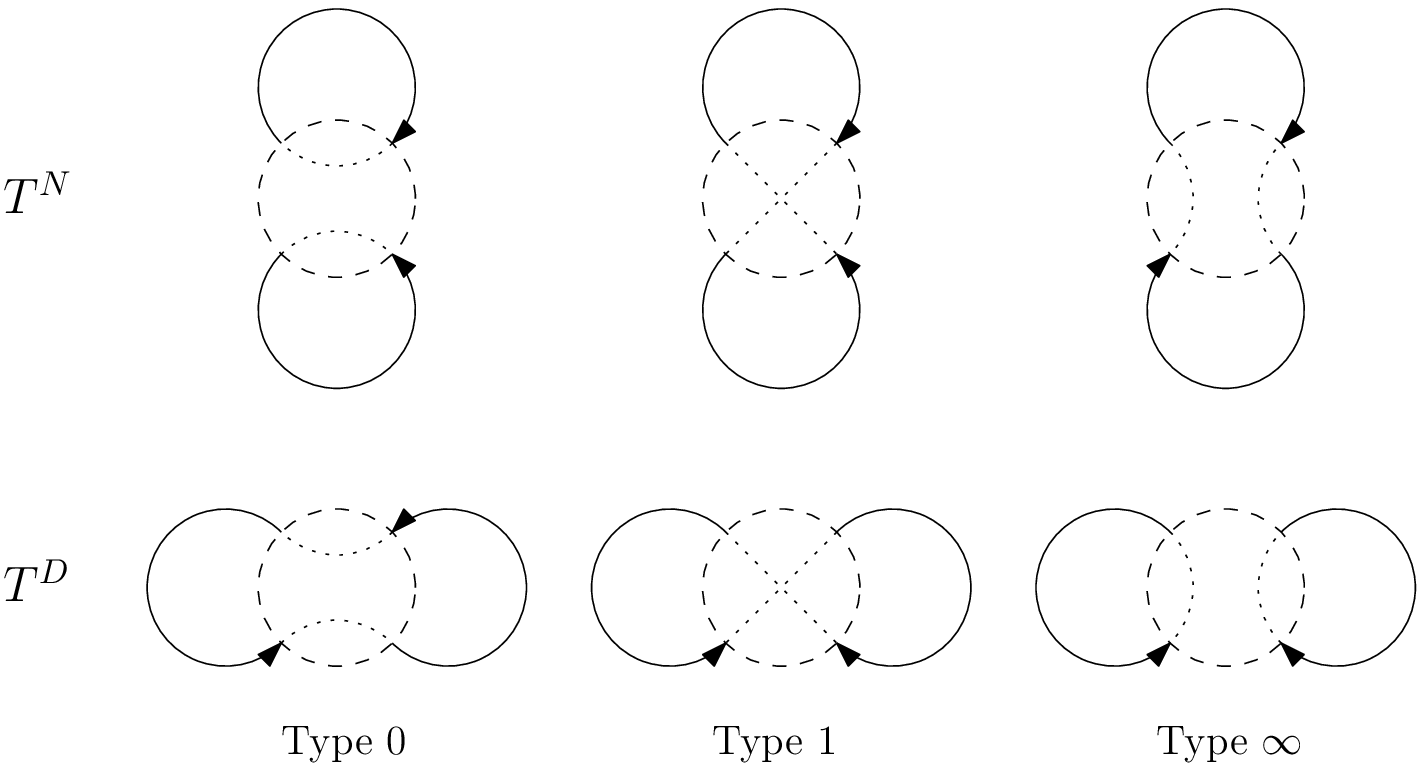}
\]
Define the {\it writhe vector} of $T$ to be $\bfw(T)=\begin{pmatrix}w(T^N)\cr{}w(T^D)\cr\end{pmatrix}$. The operations of closure introduce no additional crossings and the orientations have been chosen so that in every case contributions to $w(T^N)$ and $w(T^D)$  from crossings involving both `strands' of $T$ will be opposite. Under the operations $\ \vec{}\ $ and $+$ we have
\begin{equation}\label{eq:writherot}
\bfw(\vec{T})=\begin{pmatrix}0&1\cr1&0\cr\end{pmatrix}\bfw(T)
\end{equation}
while $T+T'$ is loopless assuming that $T,T'$ are loopless and not both of tangle type $\infty$,
\begin{equation}\label{eq:writhesum}
\bfw(T+T')=\left\{
\begin{array}{ll}
\bfw(T)+\bfw(T')&\hbox{if $\iota(T)\not=\infty$}\cr
\begin{pmatrix}0&1\cr1&0\cr\end{pmatrix}^{\iota(T')}\bfw(T)+
\begin{pmatrix}0&1\cr0&1\cr\end{pmatrix}\bfw(T')
&\hbox{if $\iota(T)=\infty$}\cr
\end{array}\right.
\end{equation}
This is obtained by verifying each possible combination of types; see the figure below where, for compactness we omit the closing arcs (where no crossings occur) and only include relevant orientations on $T+T'$ considered as a subdiagram of $(T+T')^N$ or $(T+T')^D$.
\[
\includegraphics[width=.9\textwidth]{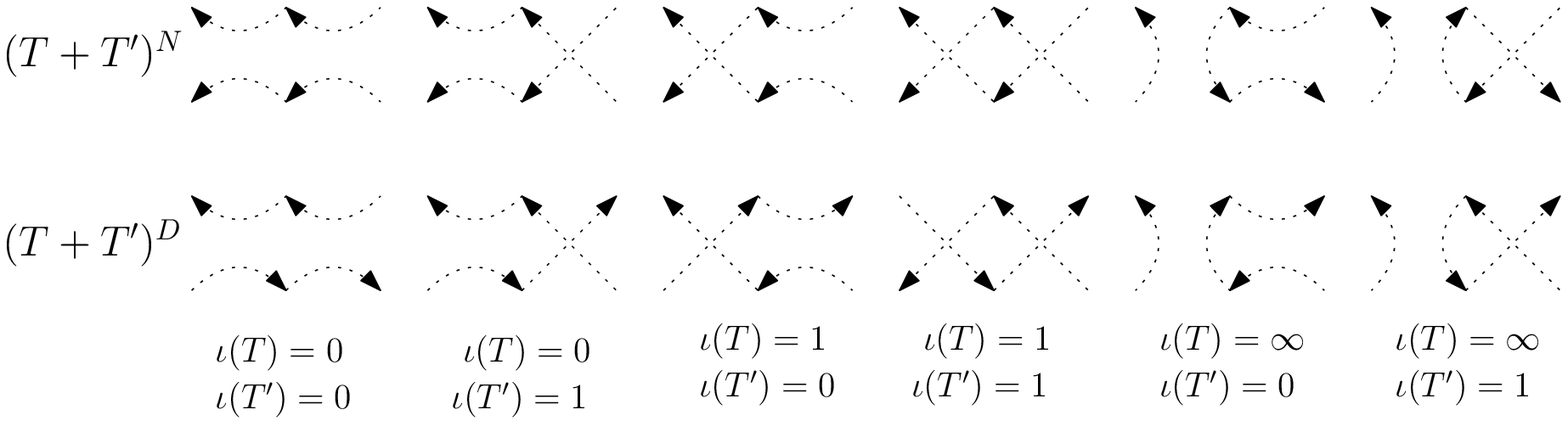}
\]

\medskip\noindent{\bf Rational tangles.} A tangle which is obtained as a sum of copies of the tangle
\[
\includegraphics[width=.15 \textwidth]{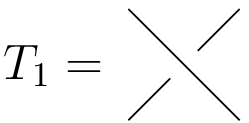}
\]
or of its rotation (denoted $T_{-1}$) through $\frac{\pi}2$, is called a {\sl horizontal twist box} $T_n$, $n\in\Z$. A tangle which can be obtained by alternating use of the operations $+T_n$ and $\ \vec{}\ $ from the tangle $T_0$ is called a {\it rational tangle}; let $\RT$ denote the set of rational tangles. By \cite{Conway} (see also \cite{KL}), a rational tangle is determined up to isotopy equivalence by the corresponding rational number; equivalently there is a well-defined bijection $r:\RT\to\Q\cup\{\infty\}$ for which
\[
r(T+T')=r(T)+r(T'),\quad{}r(\vec{T})=-r(T)^{-1}\hbox{ while }r(T_1)=1
\]
For example, $r(T_0)=0$ and $r(T_\infty)=\infty$; and so it is consistent to denote by  $T_r$ the rational tangle with fraction $r$ for arbitrary $r\in\Q\cup\{\infty\}$. The fact that the tangle type of $T_1$ is 1 combined with the transformation properties above for tangle type under $\ \vec{}\ $ and $+$ show that the tangle type of $T_r$ is the `modulo 2 evaluation of $r$', that is, if $r=\frac{p}{q}$ where $p,q$ have no common factors then
\[
\iota(T_r)=\left\{
\begin{array}{ll}
0&\hbox{if $p$ is even}\cr
1&\hbox{if $p,q$ are both odd}\cr
\infty&\hbox{if $q$ is even}\cr
\end{array}\right.
\]
In particular, for $r\in\N$, $T_r=T_1+\cdots+T_1$ ($r$ times) will be a {\it horizontal twist box} with $r$ twists while $T_{-r}$ is a horizontal twist box with $r$ twists of opposite orientation (which we will call a horizontal twist box with $-r$ twists). For $r\in\Q$ with continued fraction expansion\footnote{To fix notations, this continued fraction is the rational obtained from zero by alternately applying the operations $x\mapsto{}n_i-x$, $i=1,\ldots,k$ and reciprocation; that is, starting from zero, subtract from $n_1$, reciprocate, subtract from $n_2$, reciprocate, and so on, until the final operation of subtraction from $n_k$.} $r=n_k-\frac{1}{n_{k-1}-}\frac{1}{n_{k-2}-}\cdots\frac{1}{n_1}$   where $n_1,\ldots,n_k\in\Z$, the rational tangle $T_r$ can be obtained by $k$ iterations, starting from the twist box $T_{n_1}$, rotating, adding the twist box $T_{n_2}$, rotating and so on, up until the last addition of the twist box $T_{n_k}$. This rational tangle will also be denoted $R(n_k,n_{k-1},\ldots,n_1)$.

\begin{remark}
Note that it is more conventional to use the operation of reflection in a diagonal, as the second operation in the construction of rational tangles, rather than rotation through $\frac{\pi}2$. This is the cause of the minus signs in the continued fraction expansions that we obtain.
\end{remark}
\[
\includegraphics[width=.4\textwidth]{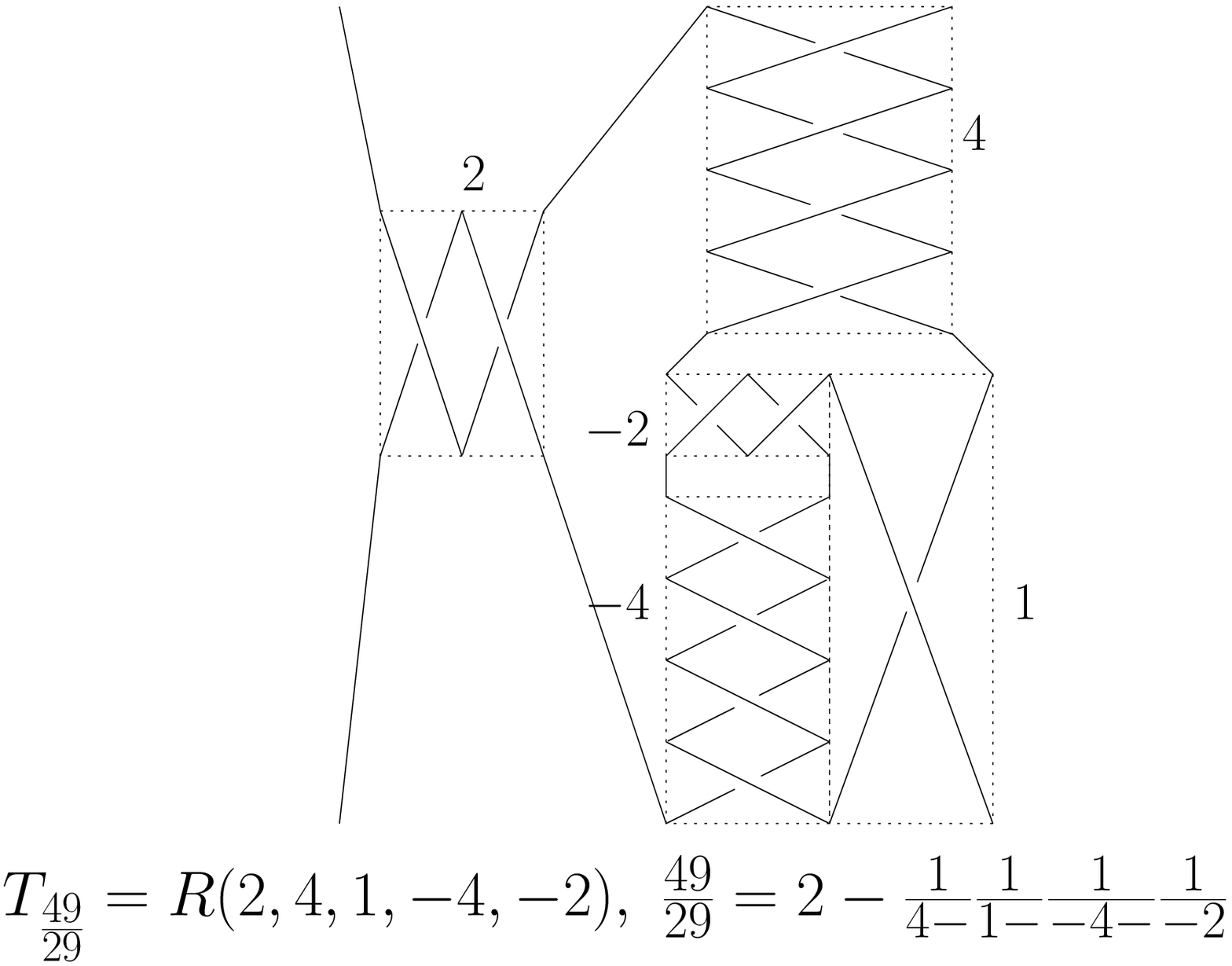}
\]

\medskip\noindent Note that the continued fraction expansion of a rational is far from unique. For example $m-\frac1{0-}\frac1{n-}\frac1x=(m+n)-\frac1x$ while $x-\frac1{0-}\frac1{-2-}\frac1{-1-}\frac1{-2}=x$.

 \begin{lemma}\label{evencontinuedfraclemma}
  Any rational in reduced form $\frac{p}{q}$ with $p$ odd and $q$ even can be written as a continued fraction of even integers and of even length, that is
  \[r=n_{k}-\frac{1}{n_{k-1}-}\frac{1}{n_{k-2}-}\cdots\frac{1}{n_1}\]
  where $k$ is even and $n_1,\ldots,n_k\in2\Z$. This expression is unique when the further condition that $n_1,\ldots,n_{k-1}\not=0$ is imposed.
 \end{lemma}
 \begin{proof}
 Suppose $r=\frac{p}{q}\in\Q$ has odd numerator and even denominator. It is certainly not an odd integer and therefore there exists a unique closest even integer, say $2m$. Then $2m-r\in(-1,1)$ is a rational with odd numerator (with absolute value less than $|q|$) and even denominator (still $q$). Put $s=\frac1{2m-r}$; this is a rational with even numerator $q$ and odd denominator (with absolute value less than $|q|$)  whose absolute value is strictly greater than 1 and so it cannot be an odd integer. Let $2n\not=0$ be the unique closest even integer to $s$ so that $2n-s\in(-1,1)$ is a rational (possibly zero) with even numerator and odd denominator (both of absolute value less than $|q|$). If $2n-s=0$ then $r=2m-\frac1{2n}$ and we are done. If not, then setting $x=\frac1{2n-s}$ we have $r=2m-\frac1{2n-}\frac1x$ where $x$ is a rational, $|x|>1$ with odd numerator and even denominator (both less than $|q|$) to which we can apply the same argument again with $x$ replacing $r$. This iterative process is finite since the denominators $q$ are of strictly decreasing absolute value. Uniqueness follows from the fact that a continued fraction with all non-zero even entries must have absolute value strictly greater than 1.
 \end{proof}

\medskip \noindent{\bf Kauffman bracket for tangles.} If one starts with a framed $(2,2)$-tangle and then chooses a planar representation (using blackboard framing), one can perform the same replacement (1) of crossings by a combination of smoothings as in the Kauffman bracket description of the previous section. The result will be an $R$-combination of $(2,2)$-tangle diagrams without crossings, each of which must be a disjoint union of some number of loops and one of the two basic $(2,2)$-tangles $T_0$, $T_\infty$.
Identifying $D\coprod{}{}U\sim{}d\cdot{}D$ results in an evaluation of $D$ as an $R$-linear combination of $T_0$, $T_\infty$, which is known as the Kauffman bracket for $(2,2)$-tangles, now a vector-valued regular isotopy invariant, see \cite{Amankwah}.

Equivalently, consider $R\T$ as an $R$-module and its submodule
 \[
 I=\Big\langle{}D-AD_\infty-A^{-1}D_0,\>D\coprod{}U-d\cdot{}D\Big\rangle
 \]
 where $D$ runs over all $(2,2)$-tangle diagrams with local smoothing moves as in the last section. The quotient is the relevant skein module, which is two-dimensional, generated by $T_0$, $T_\infty$, and the map
  \[R\T\longrightarrow{}R\T/I\cong{}R^2\]  is the Kauffman bracket of $(2,2)$-tangles, again denoted $\langle\cdot\rangle$. Thus for a $(2,2)$-tangle $T$, its Kauffman bracket $\langle{}T\rangle=\begin{pmatrix}f\cr{}g\cr\end{pmatrix}$ where $f,g\in{}R$ are such that $T\sim{}fT_0+gT_\infty$.

  By applying the corresponding operations to the equivalent tangle, it follows immediately that $\vec{T}\sim{}fT_\infty+gT_0$, $T^N\sim(df+g)U$ and $T^D\sim(f+dg)U$ so that
\begin{equation}\label{eq:Kauffmanclosure}
\langle\vec{T}\rangle=\begin{pmatrix}g\cr{}f\cr\end{pmatrix},\quad
\langle{}T^N\rangle=d\cdot{}f+g,\quad
\langle{}T^D\rangle=f+d\cdot{}g
\end{equation}
while when combined with another tangle $T'$ for which $\langle{}T'\rangle=\begin{pmatrix}f'\cr{}g'\cr\end{pmatrix}$,
\begin{equation}\label{eq:Kauffmanops}
\langle{}T*T'\rangle=\begin{pmatrix}dff'+fg'+gf'\cr{}gg'\cr\end{pmatrix},\quad
\langle{}T+T'\rangle=\begin{pmatrix}ff'\cr{}fg'+f'g+dgg'\cr\end{pmatrix}
\end{equation}

\medskip\begin{lemma}\label{bracketvectorlemma} The Kauffman bracket vector of rational tangle diagrams is given by
\[
\langle{}R(n_k,n_{k-1},\ldots,n_1)\rangle
=(iA^{-1})^{\sum\limits_sn_s}(-i)^k
\begin{pmatrix}0&1\cr1&0\cr\end{pmatrix}\bfB_{n_k}\ldots{}\bfB_{n_1}
\begin{pmatrix}1\cr0\cr\end{pmatrix}
\]
where $n_1,\ldots,n_k\in\Z$, $u=-iA^2$, $[n]=\frac{u^n-u^{-n}}{u-u^{-1}}$  and $\bfB_n\equiv\begin{pmatrix}[n]&iu^{-n}\cr{}iu^n&0\cr\end{pmatrix}$.
 \end{lemma}
\begin{remark}\label{singqnumber} In the cases $u=\pm1$, $[n]$ must be replaced by its limiting values, $n$ for $u=1$ and $(-1)^{n-1}n$ for $u=-1$.
\end{remark}
\begin{proof}
By (\ref{eq:skeinrelation}), $\langle{}T_1\rangle=\begin{pmatrix}A\cr{}A^{-1}\cr\end{pmatrix}$ and so by (\ref{eq:Kauffmanops}),
\[
\langle{}T+T_1\rangle
\!=\!\begin{pmatrix}Af\cr{}A^{-1}f\!+\!Ag\!+\!dA^{-1}g\cr\end{pmatrix}
\!=\!\begin{pmatrix}A&0\cr{}A^{-1}&-A^{-3}\cr\end{pmatrix}\!\langle{}T\rangle
\!=\!iA^{-1}\!\begin{pmatrix}u&0\cr{}-i&u^{-1}\cr\end{pmatrix}\!\langle{}T\rangle
\]
Applying this $n$ times gives
\[
\langle{}T+T_n\rangle
=(iA^{-1})^n\begin{pmatrix}u&0\cr{}-i&u^{-1}\cr\end{pmatrix}^n\langle{}T\rangle
=(iA^{-1})^n\begin{pmatrix}u^n&0\cr{}-i[n]&u^{-n}\cr\end{pmatrix}\langle{}T\rangle
\]
Meanwhile by (\ref{eq:Kauffmanclosure}),
$\langle\vec{T}\rangle=\begin{pmatrix}1&0\cr0&1\cr\end{pmatrix}\langle{}T\rangle$.
Since $T_0$ is the identity with respect to $+$ and is obtained by rotating $T_\infty$, thus $R(n_k,n_{k-1},\ldots,n_1)$ may be obtained by starting with $T_\infty$ and making $k$ iterations, each step being first a rotation and then adding $T_n$, where in turn $n=n_1,\ldots,n_k$. By definition
$\langle{}T_\infty\rangle=\begin{pmatrix}0\cr1\cr\end{pmatrix}$ and so it follows that
\begin{align*}
&\langle{}R(n_k,n_{k-1},\ldots,n_1)\rangle\cr
&=(iA^{-1})^{n_1+\cdots+n_k}
\begin{pmatrix}u^{n_k}&0\cr{}-i[n_k]&u^{-n_k}\cr\end{pmatrix}
\!\begin{pmatrix}0&1\cr1&0\cr\end{pmatrix}\cdots
\begin{pmatrix}u^n_1&0\cr{}-i[n_1]&u^{-n_1}\cr\end{pmatrix}
\!\begin{pmatrix}0&1\cr1&0\cr\end{pmatrix}
\!\begin{pmatrix}0\cr1\cr\end{pmatrix}\cr
&=(iA^{-1})^{\sum\limits_sn_s}\begin{pmatrix}0&1\cr1&0\cr\end{pmatrix}
\left(\prod\limits_{s=k}^1\begin{pmatrix}0&1\cr1&0\cr\end{pmatrix}
\begin{pmatrix}u^{n_s}&0\cr{}-i[n_s]&u^{-n_s}\cr\end{pmatrix}\right)
\begin{pmatrix}1\cr0\cr\end{pmatrix}
\end{align*}
The term inside the product can be identified as $-i\bfB_{n_s}$ and the lemma follows.
\end{proof}

\section{The Jones polynomial of a rational knot}
A {\it rational knot} is a knot which can be expressed as the closure of a rational tangle. We will use numerator closure and write $K_r\equiv(T_r)^N$ to denote the numerator closure of the rational tangle $T_r$ for $r\in\Q\cup\{\infty\}$. This will be a knot so long as $\iota(T_r)\not=0$, that is, so long as $r$ has odd numerator.

\begin{theorem}[Schubert's Theorem \cite{Schubert}, \cite{KL2}] Two rational knots $K_{\frac{p}{q}}$ and $K_{\frac{p'}{q'}}$ corresponding to reduced fractions $\frac{p}{q}$ and $\frac{p'}{q'}$ (with $p,p'>0$) are ambient isotopic if, and only if, $p=p'$ and $q\equiv{}q'^{\pm1}$ (mod $p$).
\end{theorem}

\medskip\noindent To calculate the Jones polynomial of $K_r$, pick a continued fraction expansion of $r$, say
\[r=\frac{p}{q}=n_k-\frac{1}{n_{k-1}-}\frac{1}{n_{k-2}-}\cdots\frac{1}{n_1}\]
where $n_1,\ldots,n_k\in\Z$, so that $T_r\sim{}R(n_k,\ldots,n_1)$ defines a tangle diagram for $T_r$. The Kauffman bracket of the knot diagram $R(n_k,\ldots,n_1)^N$ of $K_r$ is obtained by combining Lemma \ref{bracketvectorlemma} with $\langle{}T^N\rangle=\begin{pmatrix}d&1\cr\end{pmatrix}\langle{}T\rangle$ from (\ref{eq:Kauffmanclosure}), to give
\begin{equation}\label{eq:KauffmanQknot}
\langle{}R(n_k,n_{k-1},\ldots,n_1)^N\rangle
=(iA^{-1})^{\sum\limits_sn_s}(-i)^k
\begin{pmatrix}1&d\cr\end{pmatrix}\bfB_{n_k}\ldots{}\bfB_{n_1}
\begin{pmatrix}1\cr0\cr\end{pmatrix}
\end{equation}
By (\ref{eq:Jonespoly}), the Jones polynomial $V_L(K_r)$ is obtained from this Kauffman bracket evaluation by normalising using the writhe of the knot diagram, namely the first component of $\bfw(R(n_k,\ldots,n_1))$. This writhe vector can be computed iteratively from (\ref{eq:writherot}) and (\ref{eq:writhesum}) starting from the writhe vector of a horizontal twist box $\bfw(T_n)=\begin{pmatrix}n\cr-n\cr\end{pmatrix}$. An explicit formula is given in \cite{R} but we will not need it here. Writing everything as a function of $u$, noting that $d=-A^2-A^{-2}=i(u^{-1}-u)$ we get the following lemma.

\medskip
\begin{theorem}\label{Jonesrationalformula} The Jones polynomial of the knot obtained by numerator closure of the rational tangle $R(n_k,\ldots,n_1)$ is
 \[
V_{K_r}(t)=(-1)^{\frac14w^N+\frac14\sum_sn_s-\frac{k}2}
u^{-\frac32w^N-\frac12\sum_sn_s}
\begin{pmatrix}1&d\cr\end{pmatrix}\bfB_{n_k}\ldots{}\bfB_{n_1}
\begin{pmatrix}1\cr0\cr\end{pmatrix}
\]
where $w^N$ is the writhe of the numerator closure of the knot diagram while  $t=-u^{-2}$, $d=i(u^{-1}-u)$, $[n]=\frac{u^n-u^{-n}}{u-u^{-1}}$  and $\bfB_n\equiv\begin{pmatrix}[n]&iu^{-n}\cr{}iu^n&0\cr\end{pmatrix}$. The exponents of $-1$ and $u$ in the normalisation are integral.
\end{theorem}

\begin{remark}
Integrality of the exponents of $-1,u$ follows from the fact that the matrix product is a polynomial in $u,u^{-1}$ with integral coefficients, while the Jones polynomial of a knot is always a polynomial in $t,t^{-1}$ with integer coefficients.
\end{remark}

\begin{remark} Since $V_L$ is a topological invariant, the expression given above for $V_{K_r}(t)$ should depend on the integers $k$,$n_1,\ldots,n_k$ only via the continued fraction $r=\frac{p}{q}$. However there is no known direct formula for it in terms of $p,q$.
\end{remark}

\begin{remark} When $A=\sqrt{i}$, then $u=1$, $t=-1$ and  the Jones polynomial evaluates (up to sign) to the determinant of the knot
\[
V_{K_r}(-1)=\pm\begin{pmatrix}1&0\cr\end{pmatrix}\bfB_{n_k}\ldots{}\bfB_{n_1}
\begin{pmatrix}1\cr0\cr\end{pmatrix}
\]
where $\bfB_n\equiv\begin{pmatrix}n&i\cr{}i&0\cr\end{pmatrix}$, see Remark \ref{singqnumber}. The product  $\bfB_{n_k}\ldots\bfB_{n_1}\begin{pmatrix}1\cr0\end{pmatrix}$ evaluates via the continued fraction to $\begin{pmatrix}p\cr{}iq\end{pmatrix}$ and so we get $|V_{K_r}(-1)|=p$, as expected.
\end{remark}

%\begin{remark}
Any rational knot can be written as $K_{\frac{p}{q}}$ with $p$ odd. By Schubert's theorem, changing $q$ by a multiple of $p$ does not alter the knot and so  $q$ may be chosen to be even. By Lemma \ref{evencontinuedfraclemma}, we can now choose a continued fraction expansion of even length in even integers, that is, assume that $k$ and $n_1,\ldots,n_k$ are even. In this case, the tangle type of the intermediate rational tangles $R(n_s,\ldots,n_1)$ ($s=1,\ldots,k$) alternate between 0 ($s$ odd) and $\infty$ ($s$ even). Setting $\bfw_s=\bfw(R(n_s,\ldots,n_1))$ then by (\ref{eq:writherot}) and (\ref{eq:writhesum})
\[
\bfw_s=\bfw\left(\overset{\longrightarrow}{R(n_{s-1},\ldots,n_1)}+T_{n_s}\right)
=\left\{
\begin{array}{ll}
\begin{pmatrix}0&1\cr1&0\cr\end{pmatrix}\bfw_{s-1}
+\begin{pmatrix}n_s\cr-n_s\cr\end{pmatrix}&s\hbox{ odd}\cr
\begin{pmatrix}0&1\cr1&0\cr\end{pmatrix}\bfw_{s-1}
+\begin{pmatrix}-n_s\cr-n_s\cr\end{pmatrix}&s\hbox{ even}\cr
\end{array}
\right.
\]
from which it follows that the first entry in $\bfw_{k}$ is $w^N=-\sum\limits_{s=1}^kn_s$. The formula of Theorem \ref{Jonesrationalformula} now simplifies, while still enabling the the Jones polynomial of an arbitrary rational knot to be computed.
%\end{remark}

\begin{theorem}\label{Jonesevenformula} For even $k$ and $n_1,\ldots,n_k\in2\Z$, the Jones polynomial of the knot obtained by numerator closure of the rational tangle $R(n_k,\ldots,n_1)$ is
 \[
V_{K_r}(t)=(-1)^{\frac{k}2}
u^{\sum_sn_s}
\begin{pmatrix}1&i(u^{-1}-u)\cr\end{pmatrix}\bfB_{n_k}\ldots{}\bfB_{n_1}
\begin{pmatrix}1\cr0\cr\end{pmatrix}
\]
where $t=-u^{-2}$, $d=i(u^{-1}-u)$, $[n]=\frac{u^n-u^{-n}}{u-u^{-1}}$  and $\bfB_n\equiv\begin{pmatrix}[n]&iu^{-n}\cr{}iu^n&0\cr\end{pmatrix}$.
\end{theorem}

\section{A combinatorial formula for Jones polynomials of rational knots}

In this section we evaluate the matrix product in Theorem \ref{Jonesevenformula} and carefully investigate cancellations, to arrive at the  compact formula for the Jones polynomial of rational knots given in Theorem \ref{jonesqnewformula}.

\noindent Let $K\!=R(2m_{2k},\ldots,2m_1)^N$ be a rational knot. Write $[k]=\{1,\ldots,k\}$, not to be confused with the $q$-number of the last section.

\begin{definition}
For $k\in\N$, denote by $\s_k$ the set of all subsets $S\subset[k]$ (possibly empty) for which
\begin{itemize}
\item[(i)] $|S|$ has the same parity as $k$;
\item[(ii)] for $S\not=\emptyset$, when the elements of $S$ are listed in ascending order, the parities alternate with the first (smallest) element odd.
\end{itemize}
\end{definition}

\noindent For a set $S\in\s_k$, the complementary set $[k]\backslash{}S$ has an even number of elements, which when written in ascending order will be denoted $a_1,b_1,a_2,b_2,\ldots$.

\begin{lemma}
$\bfB_{n_k}\ldots\bfB_{n_1}\begin{pmatrix}1\cr0\cr\end{pmatrix}
=\begin{pmatrix}
\sum\limits_{S\in\s_k}i^{k-|S|}\cdot{}u^{\sum_j(n_{a_j}-n_{b_j})}
\prod\limits_{r\in{}S}[n_r]\cr
\sum\limits_{S\in\s_{k-1}}i^{k-|S|}\cdot{}u^{n_k+\sum_j(n_{a_j}-n_{b_j})}
\prod\limits_{r\in{}S}[n_r]\cr
\end{pmatrix}$
\end{lemma}
\begin{proof}
We expand the matrix product by definition as a sum over all intermediate indices of products of matrix entries. Thus the upper entry of the matrix product in the lemma is
\[
\sum\limits_{u_1,\ldots,u_{k-1}\in\{1,2\}}(\bfB_{n_k})_{u_{k},u_{k-1}}\ldots
(\bfB_{n_1})_{u_1,u_0}\eqno{(*)}
\]
where the indices $u_i$ are extended so that $u_0=u_{k}=1$. Similarly the lower entry is obtained from the same sum but where the indices are extended so that $u_0=1$, $u_k=2$.

Recall that $(\bfB_n)_{11}=[n]$, $(\bfB_n)_{12}=iu^{-n}$, $(\bfB_n)_{21}=iu^n$ while $(\bfB_n)_{22}=0$. Sequences $u_0,\ldots,u_k$ making non-zero contributions to (*) will therefore be ones in which all adjacent pairs are 11, 12 or 21; that is they consist of isolated 2's separated by strings of 1's. Such sequences starting and beginning with 1 are in bijection with elements $S\in\s_k$ where $S$ is the set of $j\in[k]$ for which $u_{j-1}=u_j=1$. Then the complementary set $[k]\backslash{}S$ when written in ascending order will have form $a_1,b_1,a_2,b_2,\ldots$ where the $a_j$ denote the positions (index values) of the isolated 2's in the sequence $u_i$ and $b_j=a_j+1$. The matrix element product in (*) corresponding to this particular sequence $u_i$ is the summand in the upper element given in the lemma.

Similarly for the lower entry in the matrix product, where now $u_k=2$. Removing this last element, relevant sequences $(u_j)$ are seen to be in bijection with $\s_{k-1}$.\end{proof}

\begin{example}
The sequence $(u_j)_{j=0}^{13}$ given by $11121211121121$ is a relevant sequence for $k=13$ and corresponds to $S=\{1,2,7,8,11\}\in\s_{13}$ which has complement in $[13]$ which reads in ascending order $3,4,5,6,9,10,12,13$ making the integers $a_j$ into $3,5,9,12$ with $b_j=a_j+1$.
\end{example}

Applying this lemma to the matrix product in Theorem \ref{Jonesevenformula} with $2k$ replacing $k$ and $2m_s$ replacing $n_s$ gives
\begin{align*}
V_{K}(t)\!&=\!
u^{2\sum\limits_sm_s}
\!\begin{pmatrix}1&i(u^{-1}\!-\!u)\cr\end{pmatrix}\!\begin{pmatrix}
\sum\limits_{S\in\s_{2k}}i^{-|S|}\cdot{}u^{2\sum\limits_j(m_{a_j}-m_{b_j})}
\prod\limits_{r\in{}S}[2m_r]\cr
\sum\limits_{S\in\s_{2k-1}}i^{-|S|}\cdot{}u^{2m_{2k}+\sum\limits_j(2m_{a_j}-2m_{b_j})}
\prod\limits_{r\in{}S}[2m_r]\cr
\end{pmatrix}\cr
&=\sum\limits_{S\in\s_{2k}}i^{-|S|}\cdot{}u^{4\sum\limits_jm_{a_j}}
\prod\limits_{r\in{}S}(u^{2m_r}[2m_r])\cr
&\qquad+i(u^{-1}-u)
 \sum\limits_{S\in\s_{2k-1}}i^{-|S|}\cdot{}u^{4m_{2k}+\sum\limits_j4m_{a_j}}
\prod\limits_{r\in{}S}(u^{2m_r}[2m_r])
\end{align*}
where in the last line we distribute the factor $u^{2\sum\limits_sm_s}$ as factors $u^{2m_s}$ amongst $s$ either in $S$ or of form $a_j,b_j$ (or, in the second line, $2k$). Recall that $u^2=-t^{-1}$ and so $u^{2m}[2m]=\frac{t^{-2m}-1}{u-u^{-1}}$. Thus
\begin{align*}
V_K(t)=&\sum\limits_{S\in\s_{2k}}\frac{(-1)^{\frac{|S|}2}}{(u-u^{-1})^{|S|}}
\cdot{}t^{-2\sum\limits_jm_{a_j}}\prod\limits_{r\in{}S}(t^{-2m_r}-1)\cr
&-\sum\limits_{S\in\s_{2k-1}}\frac{(-1)^{\frac{|S|-1}2}}{(u-u^{-1})^{|S|-1}}
\cdot{}t^{-2m_{2k}-2\sum\limits_jm_{a_j}}\prod\limits_{r\in{}S}(t^{-2m_r}-1)
\end{align*}
Note that in the first sum $|S|$ is always even while in the second sum $|S|$ is always odd. Furthermore $-(u-u^{-1})^2=t+2+t^{-1}$ while the factor $t^{-2m_{a_j}}$  can be split up as $(t^{-2m_{a_j}}-1)+1$ and more generally,
\[
t^{-2\sum\limits_jm_{a_j}}
=\prod\limits_jt^{-2m_{a_j}}
=\sum\limits_A\prod_{a\in{}A}(t^{-2m_a}-1)
\]
where the sum is over all (not necessarily proper) subsets $A\subset\{a_1,a_2,\ldots\}$. Thus
\begin{equation}\label{eq:Jonesintermediate}
\begin{array}{rl}
V_K(t)=&\sum\limits_{S\in\s_{2k}}(t+2+t^{-1})^{-\frac{|S|}2}
\cdot\sum\limits_{A\subset\{a_1,\ldots\}}\prod\limits_{r\in{}S\cup{}A}(t^{-2m_r}-1)\cr
&-\!\sum\limits_{S\in\s_{2k-1}}\!\!\!(t+2+t^{-1})^{-\frac{|S|-1}2}
\cdot\sum\limits_{A\subset\{a_1,\ldots\,2k\}}\prod\limits_{r\in{}S\cup{}A}(t^{-2m_r}-1)
\end{array}
\end{equation}
Next observe that in the first line, the possible subsets $S\cup{}A\subset[2k]$ which can occur are: the empty set and every subset whose smallest element is odd. Each such set $T$ appears precisely once and the corresponding sets $S$, $A$ can be recovered by the following algorithm. If $T=\emptyset$ then $S=A=\emptyset$. Otherwise, reading $T$ as a sequence in ascending order, split $T$ into subsequences $T_1,T_2,\ldots,T_p$ of consecutive elements of $T$ of the same parity; the first subsequence $T_1$  will be of odd parity, the second $T_2$ of even parity and so on. The largest element of $T$ will have the same parity as $p$. Then
\[
S=\left\{
\begin{array}{ll}
\big\{\max{T_j}\bigm|1\leq{}j\leq{}p\big\}&\hbox{\ if $p$ is even}\cr
\big\{\max{T_j}\bigm|1\leq{}j\leq{}p-1\big\}&\hbox{\ if $p$ is odd}\cr
\end{array}
\right.
\]
By definition $S$ has an even number of elements, $2l$ where $l\equiv\lfloor\frac{p}2\rfloor$, while, if it is non-empty then its smallest element is odd and the parities of successively increasing elements alternate.  In either case, it is easy to see that $A=T\backslash{}S$ is a subset of the $a_j$'s associated with the set $S$.

Meanwhile a similar analysis in the second line shows that the possible sets $S\cup{}A\subset[2k]$ which can occur are all those subsets $T$ with odd minimum element. To reconstruct $S$ and $A$, read $T$ in ascending order, subdivide into subsequences $T_1,\ldots,T_p$ by parity as in the previous case. Then
\[
S=\left\{
\begin{array}{ll}
\big\{\max{T_j}\bigm|1\leq{}j\leq{}p-1\big\}&\hbox{\ if $p$ is even}\cr
\big\{\max{T_j}\bigm|1\leq{}j\leq{}p\big\}&\hbox{\ if $p$ is odd}\cr
\end{array}
\right.
\]
Now $S$ has an odd number of elements, $2l+1$ where $l\equiv\lfloor\frac{p-1}{2}\rfloor$, their parities alternating when placed in ascending order, the first being odd.

Changing the order of the summations in (\ref{eq:Jonesintermediate}) now gives
\begin{align*}
V_K(t)=1+&\sum\limits_T(t+2+t^{-1})^{-\lfloor\frac{p(T)}2\rfloor}
\prod_{r\in{}T}(t^{-2m_r}-1)\cr
&-\sum\limits_T(t+2+t^{-1})^{-\lfloor\frac{p(T)-1}2\rfloor}
\prod_{r\in{}T}(t^{-2m_r}-1)
\end{align*}
where the summations are over all (non-empty) subsets $T\subset[2k]$ whose smallest element is odd. The first term comes from the empty set. Here $p(T)$ denotes the number of subsequences into which $T$ is split as above (consecutive elements of the same parity), equivalently it is one more than the number of parity changes between adjacent elements when $T$ is listed in ascending order. When $p(T)$ is odd, the contributions of $T$ to the two sums cancel. When $p(T)$ is even the contributions differ by a factor of $-(t+2+t^{-1})$. In conclusion we have the following formulation of the Jones polynomial.

\begin{theorem}\label{jonesqnewformula}
Suppose $K$ is a rational knot expressed as the numerator closure of a rational tangle made out of an even number of even-length twist boxes $R(2m_{2k},\ldots,2m_1)$.  Then the Jones polynomial of $K$ is given by
\[V_K(t)=1-(t+1+t^{-1})\sum_{T}(t+2+t^{-1})^{-\frac{p(T)}2}
\prod\limits_{r\in{}T}(t^{-2m_r}-1)
\]
where the sum is over all non-empty subsets $T\subset[2k]$ whose minimal element is odd and maximal element is even and $p(T)$ is one more than the number of parity changes between  consecutive elements when $T$ is listed in ascending order.
\end{theorem}

\begin{remark}
Note that $(t^{-2m_r}-1)$ is divisible by $t^2-1$ and thus the innermost product in the formula for $V_K(t)$ in the above theorem is divisible by $(t^2-1)^{|T|}$. Meanwhile $t+2+t^{-1}=t^{-1}(t+1)^2$ and so since $p(T)\leq|T|$, the summand is a polynomial in $t,t^{-1}$ divisible by $(1-t)^{|T|}$.
\end{remark}

\section{Templates generating Jones coincidences}

In this section we use Theorem \ref{Jonesrationalformula} and matrix identities amongst the matrices $\bfB_n$ to prove invariance of the Jones polynomial of rational knots under certain moves.

\begin{lemma}\label{almostsameJoneslemma}
If $K$ and $K'$ are two knots whose Jones polynomials differ by at most a sign and multiplication by a power of $t$, then their Jones polynomials are identical.
\end{lemma}
\begin{proof}
Since $V_K(1)=1$, thus there can be no sign difference.
Suppose otherwise that $V_K\not\equiv{}V_{K'}$; without loss of generality, assume that $V_{K'}=t^nV_K$ with $n\in\N$. By Proposition 12.5 in \cite{Jones87}, $V_K(t)-1$ is divisible by $(t-1)(t^3-1)$. Hence $V_{K'}-V_K=(t^n-1)V_K(t)$ is divisible by $(t-1)(t^3-1)$. Dividing through by $t-1$ we see that $(1+t+\cdots+t^{n-1})V_K(t)$ is divisible by $t^3-1$ which contradicts the fact that its evaluation at $t=1$ is $n$.
\end{proof}

Combining with Theorem~\ref{Jonesrationalformula}, this means that for identification of Jones polynomials of two rational knots, it is sufficient to prove that the corresponding matrix product agrees (even up to a sign and power of $u$).

\begin{lemma}\label{Bproperties}
The matrices $\bfB_n\equiv\begin{pmatrix}[n]&iu^{-n}\cr{}iu^n&0\cr\end{pmatrix}\in{}SL_2$ have the following elementary properties. Here $d=i(u^{-1}-u)$.
\begin{itemize}
    \item[(i)] $\bfB_n^*=-\bfB_{-n}$ where $^*$ is conjugate transpose (conjugate being defined by $i\longmapsto-i$, $u\longmapsto{}u$).
    \item[(ii)] Products of $\bfB_n$ matrices are matrices of the form $\bfC=\begin{pmatrix}\alpha&i\beta\cr{}i\gamma&\delta\cr\end{pmatrix}$ where $\alpha,\beta,\gamma,\delta\in\Z[u,u^{-1}]$ with $\alpha\delta+\beta\gamma=1$;  in particular $\bfC^{-1}=-\bfB_0\bfC^*\bfB_0$,
    $\begin{pmatrix}1&0\end{pmatrix}\bfC^*\begin{pmatrix}1\cr0\end{pmatrix}
    =\begin{pmatrix}1&0\end{pmatrix}\bfC\begin{pmatrix}1\cr0\end{pmatrix}$
    and $\begin{pmatrix}1&d\end{pmatrix}\bfC^*\begin{pmatrix}1\cr-d\end{pmatrix}
    =\begin{pmatrix}1&d\end{pmatrix}\bfC\begin{pmatrix}1\cr-d\end{pmatrix}$.
    \item[(iii)] $\bfB_n\bfB_0\bfB_m=-\bfB_{n+m}$.
    \item[(iv)]$id\cdot\bfB_n
    =u^n\begin{pmatrix}1\cr-d\end{pmatrix}\begin{pmatrix}1&0\end{pmatrix}
    -u^{-n}\begin{pmatrix}1\cr0\end{pmatrix}\begin{pmatrix}1&d\end{pmatrix}$.
    \item[(v)] $\begin{pmatrix}1&d\end{pmatrix}\prod\limits_{j=1}^k\bfB_{n_j}
    \begin{pmatrix}1\cr0\end{pmatrix}
    =\begin{pmatrix}1&d\end{pmatrix}\prod\limits_{j=k}^1\bfB_{n_j}
    \begin{pmatrix}1\cr0\end{pmatrix}$.
    \item[(vi)] $\begin{pmatrix}1&d\end{pmatrix}\prod\limits_{j=1}^k\bfB_{n_j}
    \begin{pmatrix}1\cr-d\end{pmatrix}
    =(1-d^2)\cdot\begin{pmatrix}1&0\end{pmatrix}\prod\limits_{j=k}^1\bfB_{n_j}
    \begin{pmatrix}1\cr0\end{pmatrix}$.
    \item[(vii)]
    $\left(\begin{pmatrix}1&d\end{pmatrix}\bfC
    \begin{pmatrix}1\cr0\end{pmatrix}\right)_{u\rightarrow{}u^{-1}}
    \!\!\!\!\!\!=\begin{pmatrix}1&d\end{pmatrix}\bfC^*\begin{pmatrix}1\cr0\end{pmatrix}
    =\begin{pmatrix}1&0\end{pmatrix}\bfC\begin{pmatrix}1\cr-d\end{pmatrix}$
    for any matrix product $\bfC$ of matrices $\bfB_n$.
\end{itemize}
\end{lemma}
\begin{proof}
The first four properties are immediate from the definition of $\bfB_n$. Properties (v), (vi) follow by induction on $k$; the inductive step uses (iv). Thus, for $k=0$, an empty product is the identity matrix and (v),(vi) hold. Meanwhile, assuming both (v) and (vi) for $k$ and writing $\bfC=\prod\limits_{j=1}^k\bfB_{n_j}$,  $\backC=\prod\limits_{j=k}^1\bfB_{n_j}$ while $m=n_{k+1}$ we have by (iv)
\begin{align*}
&id\cdot\begin{pmatrix}1&d\end{pmatrix}\bfC\bfB_m\begin{pmatrix}1\cr0\end{pmatrix}\cr
&\overset{(iv)}{=}u^m\begin{pmatrix}1&d\end{pmatrix}\bfC\begin{pmatrix}1\cr-d\end{pmatrix}\begin{pmatrix}1&0\end{pmatrix}\begin{pmatrix}1\cr0\end{pmatrix}
-u^{-m}\begin{pmatrix}1&d\end{pmatrix}\bfC\begin{pmatrix}1\cr0\end{pmatrix}\begin{pmatrix}1&d\end{pmatrix}\begin{pmatrix}1\cr0\end{pmatrix}\cr
&=u^m\begin{pmatrix}1&d\end{pmatrix}\bfC\begin{pmatrix}1\cr-d\end{pmatrix}
-u^{-m}\begin{pmatrix}1&d\end{pmatrix}\bfC\begin{pmatrix}1\cr0\end{pmatrix}\cr
&\overset{(v),(vi)}{=}u^m(1-d^2)\begin{pmatrix}1&0\end{pmatrix}\backC\begin{pmatrix}1\cr0\end{pmatrix}
-u^{-m}\begin{pmatrix}1&d\end{pmatrix}\backC\begin{pmatrix}1\cr0\end{pmatrix}\cr
&=u^m\begin{pmatrix}1&d\end{pmatrix}\begin{pmatrix}1\cr-d\end{pmatrix}\begin{pmatrix}1&0\end{pmatrix}\backC\begin{pmatrix}1\cr0\end{pmatrix}
-u^{-m}\begin{pmatrix}1&d\end{pmatrix}\begin{pmatrix}1\cr0\end{pmatrix}
\begin{pmatrix}1&d\end{pmatrix}\backC\begin{pmatrix}1\cr0\end{pmatrix}\cr
&\overset{(iv)}{=}id\cdot\begin{pmatrix}1&d\end{pmatrix}\bfB_m\backC\begin{pmatrix}1\cr0\end{pmatrix}
\end{align*}
which is (v) with $k+1$ replacing $k$. Similarly for (vi).

\noindent Finally for (vii), note that $\bfB_n^T=(\bfB_n)_{u\rightarrow{}u^{-1}}$ from which it follows that $\backC=\bfC^T_{u\rightarrow{}u^{-1}}$. Since $\overline{d}=-d=d_{u\rightarrow{}u^{-1}}$, by (v)
\[\left[\begin{pmatrix}1&d\end{pmatrix}\!\bfC\!\begin{pmatrix}1\cr0\end{pmatrix}
\right]_{u\rightarrow{}u^{-1}}
\!\!\!\!\!\!\!\!=\!\left[\begin{pmatrix}1&d\end{pmatrix}\!\backC\!\begin{pmatrix}1\cr0\end{pmatrix}\right]_{u\rightarrow{}u^{-1}}
\!\!\!\!\!\!=\!\begin{pmatrix}1&\!\!-d\end{pmatrix}\!\bfC^T\!\begin{pmatrix}1\cr0\end{pmatrix}
\!\overset{T}{=}\!\!\begin{pmatrix}1&0\end{pmatrix}\!\bfC\!\begin{pmatrix}1\cr-d\end{pmatrix}
\]
\end{proof}

\begin{lemma}[\bf Key Lemma]\label{keylemma}
Suppose that $\bfU_j,\bfV_j$ $(0\leq{}j\leq{}k)$ are products of matrices $\bfB_n$ and $n_1,\ldots,n_k\in\Z$.

\noindent (1) If
 \begin{equation*}\begin{pmatrix}1&d\end{pmatrix}\bfU_j\begin{pmatrix}1\cr0\end{pmatrix}
 =\begin{pmatrix}1&d\end{pmatrix}\bfV_j\begin{pmatrix}1\cr0\end{pmatrix}
 \qquad\forall{}j
 \end{equation*}
  \begin{equation*}\begin{pmatrix}1&d\end{pmatrix}\bfU_j\begin{pmatrix}1\cr-d\end{pmatrix}\begin{pmatrix}1&0\end{pmatrix}\bfU_l\begin{pmatrix}1\cr0\end{pmatrix}
 =\begin{pmatrix}1&d\end{pmatrix}\bfV_j\begin{pmatrix}1\cr-d\end{pmatrix}\begin{pmatrix}1&0\end{pmatrix}\bfV_l\begin{pmatrix}1\cr0\end{pmatrix}
 \quad\forall{}j\not={}l
 \end{equation*}
 then
 \[
 \begin{pmatrix}1&d\end{pmatrix}
 \bfU_0\bfB_{n_1}\bfU_1\cdots\bfU_{k-1}\bfB_{n_k}\bfU_k
 \!\begin{pmatrix}1\cr0\end{pmatrix}
= \begin{pmatrix}1&d\end{pmatrix}
 \bfV_0\bfB_{n_1}\bfV_1\cdots\bfV_{k-1}\bfB_{n_k}\bfV_k
 \!\begin{pmatrix}1\cr0\end{pmatrix}
 \]

\noindent (2) If
 \begin{equation*}\begin{pmatrix}1&d\end{pmatrix}\bfU_j\begin{pmatrix}1\cr0\end{pmatrix}
 =\begin{pmatrix}1&d\end{pmatrix}\bfV_j\begin{pmatrix}1\cr0\end{pmatrix}
 \qquad\forall{}j
 \end{equation*}
  \begin{equation*}\begin{pmatrix}1&d\end{pmatrix}\bfU_l\begin{pmatrix}1\cr-d\end{pmatrix}\begin{pmatrix}1&0\end{pmatrix}\bfU_j\begin{pmatrix}1\cr0\end{pmatrix}
 =\begin{pmatrix}1&d\end{pmatrix}\bfV_j\begin{pmatrix}1\cr-d\end{pmatrix}\begin{pmatrix}1&0\end{pmatrix}\bfV_l\begin{pmatrix}1\cr0\end{pmatrix}
 \quad\forall{}j\not={}l
 \end{equation*}
 then
 \[
 \begin{pmatrix}1&d\end{pmatrix}
 \bfU_0\bfB_{n_1}\bfU_1\cdots\bfU_{k-1}\bfB_{n_k}\bfU_k
 \begin{pmatrix}1\cr0\end{pmatrix}
\!=\! \begin{pmatrix}1&d\end{pmatrix}
 \bfV_k\bfB_{n_k}\bfV_{k-1}\cdots\bfV_1\bfB_{n_1}\bfV_0
 \begin{pmatrix}1\cr0\end{pmatrix}
 \]
\end{lemma}
\begin{proof}
The proof proceeds by replacing each $\bfB_{n_j}$ in the two sides of the equality to be proved by a combination of matrix products using (iv). This results in $(id)^k\cdot\begin{pmatrix}1&d\end{pmatrix}
 \bfU_0\bfB_{n_1}\bfU_1\cdots\bfU_{k-1}\bfB_{n_k}\bfU_k
 \begin{pmatrix}1\cr0\end{pmatrix}$ being replaced by a sum of $2^k$ terms, each of which is of the form $\prod\limits_j\epsilon_j\cdot{}u^{\sum_j\epsilon_jn_j}$ (for $\epsilon_j\in\{1,-1\}$) times a product of $(k+1)$ scalars, each of which is of one of the forms
 \[
 \begin{pmatrix}1&d\end{pmatrix}\bfU_j\begin{pmatrix}1\cr0\end{pmatrix},\quad
 \begin{pmatrix}1&d\end{pmatrix}\bfU_j\begin{pmatrix}1\cr-d\end{pmatrix},\quad
 \begin{pmatrix}1&0\end{pmatrix}\bfU_j\begin{pmatrix}1\cr0\end{pmatrix},\quad
 \begin{pmatrix}1&0\end{pmatrix}\bfU_j\begin{pmatrix}1\cr-d\end{pmatrix}
 \]
according as $(\epsilon_j,\epsilon_{j+1})=(-1,-1),(-1,1),(1,-1),(1,1)$, respectively, where we extend $\epsilon_j$ to $j=0,k+1$ by setting $\epsilon_0=\epsilon_{k+1}=-1$. Similarly for the right hand side of the equality to be proved. On comparison, the first condition ensures that $\begin{pmatrix}1&d\end{pmatrix}\bfU_j\begin{pmatrix}1\cr0\end{pmatrix}
=\begin{pmatrix}1&d\end{pmatrix}\bfV_j\begin{pmatrix}1\cr0\end{pmatrix}$. By Lemma \ref{Bproperties}(vii), replacing $u\rightarrow{}u^{-1}$ gives also
$\begin{pmatrix}1&0\end{pmatrix}\bfU_j\begin{pmatrix}1\cr-d\end{pmatrix}
=\begin{pmatrix}1&0\end{pmatrix}\bfV_j\begin{pmatrix}1\cr-d\end{pmatrix}$. Since $\epsilon_0=\epsilon_{k+1}$, adjacent unequal terms in the sequence $\epsilon_j$ come in pairs $(-1,1)$ followed (possibly after a string of 1's) by $(1,-1)$; the second condition allows the product of a pair of scalars associated with $\bfU_j$'s to be equated with the corresponding product of scalars associated with $\bfV_j$'s (with order reversed in the case of (2)).
\end{proof}

\begin{definition}
Suppose that $\bfn\in\Z^r$ is an integer sequence. Denote by $-\bfn$ and $\backn$ respectively, the sequences obtained from $\bfn$ by reversing all the signs and by reversing order of the sequence, respectively. Define $\bfn^*=-\overset{\leftarrow}{\bfn}$.
\end{definition}

\begin{theorem}[\bf Template I]\label{firsttemplate} If two rational knots have integer sequence presentations related by the move
\[
(\bfn^{\epsilon_0},d_1,\bfn^{\epsilon_1},\ldots,d_k,\bfn^{\epsilon_k})
\longrightarrow
(\bfn^{\epsilon_k},d_k,\bfn^{\epsilon_{k-1}},\ldots,d_1,\bfn^{\epsilon_0})
\]
for some $d_1,\ldots,d_k\in\Z$, $\epsilon_0,\ldots,\epsilon_k\in\{1,*\}$ and integer sequence $\bfn$,
then their Jones polynomials coincide.
\end{theorem}
\begin{proof}
Let $\bfC=\bfB_{n_1}\cdots\bfB_{n_r}$ where $\bfn=(n_1,\ldots,n_r)$.
By Lemma \ref{Bproperties}(i), $\bfC^*=(-1)^r\bfB_{-n_r}\cdots\bfB_{-n_1}$.  The theorem follows from the Key Lemma (2) with $\bfU_j=\bfV_j=\bfC^{\epsilon_j}$. The first condition is automatically satisfied since $\bfU_j=\bfV_j$. The second condition follows from the fact that in any case $\bfU_j,\bfV_j\in\{\bfC,\bfC^*\}$ so that
$\begin{pmatrix}1&d\end{pmatrix}\bfU_l\begin{pmatrix}1\cr-d\end{pmatrix}
=\begin{pmatrix}1&d\end{pmatrix}\bfV_j\begin{pmatrix}1\cr-d\end{pmatrix}$
and $\begin{pmatrix}1&0\end{pmatrix}\bfU_j\begin{pmatrix}1\cr0\end{pmatrix}
 =\begin{pmatrix}1&0\end{pmatrix}\bfV_l\begin{pmatrix}1\cr0\end{pmatrix}$
by Lemma \ref{Bproperties}(ii).
\end{proof}

\begin{example}
Using $\bfn=(2,3)$, $k=2$, $\epsilon_0=\epsilon_1=\epsilon_2=1$, $d_1=0$, $d_2=-1$ in the first template leads to a Jones rational coincidence between the integer sequences $(2,3,0,2,3,-1,2,3)$ and $(2,3,-1,2,3,0,2,3)$. The associated continued fractions are
\[
2-\frac{1}{3-}\frac{1}{0-}\frac{1}{2-}\frac{1}{3-}\frac{1}{-1-}\frac{1}{2-}\frac13=\frac{245}{137},\qquad
2-\frac{1}{3-}\frac{1}{-1-}\frac{1}{2-}\frac{1}{3-}\frac{1}{0-}\frac{1}{2-}\frac13=\frac{245}{142}
\]
giving coincidence of the Jones polynomials of $K_\frac{245}{137}$ and $K_\frac{245}{142}$.
\end{example}

\begin{example}\label{tripleex}
Using $\bfn=(2,-2)$, $k=2$, $\epsilon=(1,1,1)$, $d_1=1$, $d_2=4$ in the first template generates the Jones rational coincidence given by the integer sequences \[(2,-2,1,2,-2,4,2,-2)\leftrightarrow(2,-2,4,2,-2,1,2-,2)\] that is between $K_\frac{495}{218}$ and $K_\frac{495}{203}$. On the other hand, the first template with $\bfn=(2,5,1,0)$, $k=2$, $\epsilon=(1,1,*)$, $d_1=0$, $d_2=-3$ generates the Jones  rational coincidence given by the integer sequences \[(2,5,1,0,0,2,5,1,0,-3,0,\!-1,\!-5,\!-2)\leftrightarrow (0,\!-1,\!-5,\!-2,-3,2,5,1,0,0,2,5,1,0)\] that is between the rational knots $K_\frac{495}{302}$ and $K_\frac{495}{383}$. Since $302\equiv218^{-1}$ (mod 495), this gives a Jones rational triple coincidence $K_\frac{495}{218}=K_\frac{495}{302}$, $K_\frac{495}{203}$ and $K_\frac{495}{383}$.
\end{example}

\bigskip\noindent Let $\backn$, $\bfn$ be denoted by $\overset{\epsilon}{\bfn}$ where $\epsilon=\leftarrow,\rightarrow$, respectively.

\begin{theorem}[\bf Template II]\label{secondtemplate}
If two rational knots have integer sequence presentations related by the move
\[
(\overset{\epsilon_0}{\bfn},m_0,-\overset{\phi_0}{\bfn},d_1,\ldots,d_k,\overset{\epsilon_k}{\bfn},m_k,-\overset{\phi_k}{\bfn})
\longrightarrow
(\overset{\epsilon_0}{\bfn^*},m_0,-\overset{\phi_0}{\bfn^*},d_1,\ldots,d_k,\overset{\epsilon_k}{\bfn^*},m_k,-\overset{\phi_k}{\bfn^*})
\]
for some $m_0,\ldots,{}m_k,d_1,\ldots,d_k\in\Z$, $\epsilon_0,\ldots,\epsilon_k,\phi_0,\ldots,\phi_k\in\{\leftarrow,\rightarrow\}$ and integer sequence $\bfn$,
then their Jones polynomials coincide.
\end{theorem}
\begin{proof}
 The template follows from the Key Lemma (1) where $\bfU_j$ and $\bfV_j$ are the products of $\bfB_n$ matrices associated with the integer sequences $(\overset{\epsilon_j}{\bfn},m_j,-\overset{\phi_j}{\bfn})$ and
$(\overset{\epsilon_j}{\bfn^*},m_j,-\overset{\phi_j}{\bfn^*})$. It remains to verify the two conditions in Lemma \ref{keylemma}(1). For the first one, either $\epsilon_j=\phi_j$ in which case
\[
\begin{pmatrix}1&d\end{pmatrix}\bfC\bfB_m\bfC^-\begin{pmatrix}1\cr0\end{pmatrix}
    =\begin{pmatrix}1&d\end{pmatrix}\bfC^*\bfB_m\overset{\leftarrow}{\bfC}    \begin{pmatrix}1\cr0\end{pmatrix}
\]
follows from Lemma \ref{Bproperties}(v), since the sequences $\bfC\bfB_n\bfC^-$, $\bfC^*\bfB_n\overset{\leftarrow}{\bfC}$ are exact reversals of each other; or $\epsilon_j$, $\phi_j$ are in opposite directions, in which case,
\[
\begin{pmatrix}1&d\end{pmatrix}\bfC\bfB_m\bfC^*\begin{pmatrix}1\cr0\end{pmatrix}
    =\begin{pmatrix}1&d\end{pmatrix}\bfC^*\bfB_m\bfC    \begin{pmatrix}1\cr0\end{pmatrix}
\]
follows from Theorem \ref{firsttemplate} with $k=1$. The second condition to be verified is that
\begin{align*}
&\begin{pmatrix}1&d\end{pmatrix}\bfB(\overset{\epsilon}\bfn)\bfB_m\bfB(-\overset{\phi}{\bfn})
\begin{pmatrix}1\cr-d\end{pmatrix}\cdot
\begin{pmatrix}1&0\end{pmatrix}\bfB(\overset{\epsilon'}\bfn)\bfB_k\bfB(-\overset{\phi'}{\bfn})
\begin{pmatrix}1\cr0\end{pmatrix}\cr
 &=\begin{pmatrix}1&d\end{pmatrix}\bfB(\overset{\epsilon}{\bfn^*})\bfB_m\bfB(-\overset{\phi}{\bfn^*})
\begin{pmatrix}1\cr-d\end{pmatrix}\cdot
\begin{pmatrix}1&0\end{pmatrix}\bfB(\overset{\epsilon'}{\bfn^*})\bfB_k\bfB(-\overset{\phi'}{\bfn^*})
\begin{pmatrix}1\cr0\end{pmatrix}
 \end{align*}
where $\bfB(\bfn)\equiv\bfB_{n_1}\cdots\bfB_{n_r}$. This is checked using Lemma \ref{Bproperties}(iv) to expand out the $\bfB_m$ and $\bfB_k$ terms on both sides. Thus $(-d^2)$ times the expression on either side of the equality to be proved becomes a sum  of four terms with external factors of $\pm{}u^{\pm{}m\pm{}k}$. The four matching individual terms are found to be equal, each being a product of four scalars. For example, the $-u^{m-k}$ term on the left hand side is
\[
\begin{pmatrix}1&d\end{pmatrix}\bfB(\overset{\epsilon}\bfn)
  \begin{pmatrix}1\cr-d\end{pmatrix}\cdot\begin{pmatrix}1&0\end{pmatrix}\bfB(-\overset{\phi}{\bfn})
\begin{pmatrix}1\cr-d\end{pmatrix}\cdot
\begin{pmatrix}1&0\end{pmatrix}\bfB(\overset{\epsilon'}\bfn)
 \begin{pmatrix}1\cr0\end{pmatrix}\cdot\begin{pmatrix}1&d\end{pmatrix}\bfB(-\overset{\phi'}{\bfn})
\begin{pmatrix}1\cr0\end{pmatrix}
\]
while on the right hand side it is the same with $\bfn^*$ replacing $\bfn$. The matching of first and third factors is via
\[
\begin{pmatrix}1&d\end{pmatrix}\bfC\begin{pmatrix}1\cr-d\end{pmatrix}
=\begin{pmatrix}1&d\end{pmatrix}\bfC^*\begin{pmatrix}1\cr-d\end{pmatrix},\quad
\begin{pmatrix}1&0\end{pmatrix}\bfC\begin{pmatrix}1\cr0\end{pmatrix}
=\begin{pmatrix}1&0\end{pmatrix}\bfC^*\begin{pmatrix}1\cr0\end{pmatrix}
\]
from Lemma \ref{Bproperties}(ii) (where $\bfC=\bfB(\overset{\epsilon}\bfn),\bfB(\overset{\epsilon'}\bfn)$ respectively). The second factor can be transformed into a type similar to the fourth using Lemma \ref{Bproperties}(vii); the product of the second and fourth factors is now
\[
\begin{pmatrix}1&d\end{pmatrix}\bfB(-\overset{\phi}{\bfn^*})
\begin{pmatrix}1\cr0\end{pmatrix}
\cdot\begin{pmatrix}1&d\end{pmatrix}\bfB(-\overset{\phi'}{\bfn})
\begin{pmatrix}1\cr0\end{pmatrix}
\]
which by Lemma \ref{Bproperties} is independent of $\phi,\phi'$ and is thus unchanged when $\bfn$ is replaced by $\bfn^*$  (the two factors interchanging).
\end{proof}

\begin{example}
 Take $\bfn=(4,2)$, $k=1$, $d_1=0$, $m=(-1,0)$ with $\epsilon=\phi=(\rightarrow,\rightarrow)$. The second template now gives a Jones rational coincidence between the integer sequences $(4,2,-1,-4,-2,0,4,2,0,-4,-2)$, $(-2,-4,-1,2,4,0,-2,-4,0,2,4)$, that is between the rational knots $K_\frac{329}{89}$ and $K_\frac{329}{-193}$.
\end{example}

\begin{remark}
All the explicit examples of Jones rational equivalences given in the text are not mutants as can be verified by computing their HOMFLYPT polynomial (\cite{HOMFLY}, \cite{PT}). By using the freedoms in the Templates of Theorems \ref{firsttemplate}, \ref{secondtemplate}, it can be seen that arbitrarily large (but necessarily finite) families of rational knots sharing the same Jones polynomial can be constructed; this was already known \cite{Kanenobu86}, where a special case of the construction given here was presented.
\end{remark}

\section{Pivoting pairs}
Out of the 80,317 rational knots up to determinant 899, there were found in \cite{R}, 223 Jones rational coincidences (up to mirror image), two of which are amphicheiral pairs and three of which are triplets. Out of the 220 `simple' coincidences, 144 can be identified as examples of Template I and 177 as examples of Template II, while 108 are examples of both. Note that since there are an infinite number of continued fraction representations of a rational corresponding to a particular rational knot, it is possible that even some of those identified as examples of one template, may be as-yet-unfound examples of the other template. On the other hand, of the three triple coincidences, these can be established by an `intersecting pair' of coincidences as in Example \ref{tripleex}; in two cases, both come from Template I while in the third case one comes from Template I and the other from Template II.

\begin{example}
The first Jones rational coincidence (smallest determinant) is between $K_\frac{49}{22}=10_{35}$ and $K_\frac{49}{36}=10^*_{22}$ (see \cite{atlas}). It appears as an example of Template I with $\bfn=(2,4)$, $k=1$, $d_1=1$, $\epsilon=(1,*)$ which gives the integer sequences $(2,4,1,-4,-2)$ and $(-4,-2,1,2,4)$; this is trivially also an example of Template II with $\bfn=(2,4)$, $k=0$, $m_0=1$, $\epsilon_0=\rightarrow$ and $\phi_0=\leftarrow$.
\end{example}

\begin{example}
Here is an example of a Jones rational equivalence which can be explained both by Template I
and Template II, but where the generating continued fraction expansions are quite different in
the two cases. Template I applied with $\bfn=(2,6)$, $k=2$, $d_1=0$, $d_2=1$,
$\epsilon=(1,1,1)$ gives the integer sequences $(2,6,0,2,6,1,2,6)$ and $(2,6,1,2,6,0,2,6)$
whose continued fractions evaluate to $\frac{275}{147}$ and $\frac{275}{158}$ respectively.
Template II applied with $\bfn=(4,3)$, $k=2$, ${\bf m}=(0,0,0)$, $d_1=0$, $d_2=-1$,
$\epsilon=\phi=(\rightarrow,\rightarrow,\rightarrow)$ gives the integer sequences
$(4,3,0,-4,-3,0,4,3,0,-4,-3,-1,4,3,0,-4,-3)$ and $(-3,-4,0,3,4,0,-3,-4,0,3,4,-1,-3,-4,0,3,4)$
whose continued fraction evaluations are $\frac{275}{58}$ and $-\frac{275}{117}$,
respectively. Note that $147\cdot58\equiv1\ (275)$ while $158\equiv-117\ (275)$ so that these
templates both give rise to the same Jones rational coincidence.
\end{example}

This leaves 11 Jones polynomial coincidences amongst rational knots with determinant $<900$ which are still unaccounted for by the templates of Theorems \ref{firsttemplate}, \ref{secondtemplate}.

\begin{example}\label{firstpivotcase}
 The first case of a Jones rational coincidence which doesn't seem to appear as an example of either template from \S6 has determinant 377, namely $K_\frac{377}{70}$ and $K_\frac{377}{278}$. Integer sequences for these rational knots are \[(5,-2,-4,-4,0,5,-3,-6,-4),\quad(5,-3,-6,-4,0,5,-2,-4,-4)\]
\end{example}

\noindent This first example was what led to the definition of {\sl pivoting pairs} below. Before this, we make some observations about the form of a matrix which can be expressed as a product of $\bfB_n$ matrices.

\begin{remark}
Suppose that $\bfC$ is a product of matrices $\bfB_n$. Since $\det\bfB_n=1$ thus $\det\bfC=1$.  Recall that $\overset{\leftarrow}{\bfC}=\bfC^T_{u\rightarrow{}u^{-1}}$ and so by
 Lemma \ref{Bproperties}(v),(vi),
\begin{equation}\label{eq:Cprop}\begin{array}{rcl}
\begin{pmatrix}1&d\end{pmatrix}\bfC\begin{pmatrix}1\cr0\end{pmatrix}
    &=&\begin{pmatrix}1&d\end{pmatrix}\bfC^T_{u\rightarrow{}u^{-1}}\begin{pmatrix}1\cr0\end{pmatrix}\cr
\begin{pmatrix}1&d\end{pmatrix}\bfC\begin{pmatrix}1\cr-d\end{pmatrix}
&=&(1-d^2)\cdot\begin{pmatrix}1&0\end{pmatrix}\bfC^T_{u\rightarrow{}u^{-1}}\begin{pmatrix}1\cr0\end{pmatrix}
\end{array}
\end{equation}
Matrices $\bfC=\begin{pmatrix}\alpha&i\beta\cr{}i\gamma&\delta\end{pmatrix}$ satisfying (\ref{eq:Cprop}) have
\[\alpha+id\gamma=\alpha'+id\beta',\quad\alpha-id\beta+id\gamma-d^2\delta=(1-d^2)\alpha'\]
where the prime denotes the image under $u\rightarrow{}u^{-1}$. It follows that such a matrix is determined by the entries in the first row, $\gamma=\beta'+\frac{\alpha'-\alpha}{u-u^{-1}}$ and $\delta=\alpha'+\frac{\beta-\beta'}{u-u^{-1}}$.
\begin{equation}\label{eq:Cform}
\bfC=\begin{pmatrix}
\alpha&i\beta\cr
i\beta'+i\frac{\alpha'-\alpha}{u-u^{-1}}&\alpha'+\frac{\beta-\beta'}{u-u^{-1}}
\end{pmatrix}
\hbox{ with }\alpha\alpha'+\beta\beta'+\frac{\alpha'\beta-\alpha\beta'}{u-u^{-1}}=1
\end{equation}
Matrices of the form (\ref{eq:Cform}) where $\alpha,\beta\in\Z[u,u^{-1}]$ form a group, $\gpC$, containing as a subgroup $\gpB$, the set of those matrices which can be written as a product of $\bfB_n$ matrices.
Such matrices which arise as product of $\bfB_n$ matrices will have either $\alpha\in\Z[t,t^{-1}]$, $\beta\in{}u\Z[t,t^{-1}]$ or $\alpha\in{}u\Z[t,t^{-1}]$, $\beta\in\Z[t,t^{-1}]$, according as the number of even indices ($n$ even) in the product is even or odd, respectively. This follows from the same fact about individual $\bfB_n$ matrices and its invariance under matrix product.
   Observe that the condition on $\alpha,\beta$ in (\ref{eq:Cform}) can be reformulated as
\[
(\alpha+(u^{-1}-u)\beta)(\alpha'+(u-u^{-1})\beta')=(u^2-1+u^{-2})\alpha\alpha'-(u-u^{-1})^2
\]
while the second factor on the left hand side is, by Theorem \ref{Jonesrationalformula}, the Jones polynomial of the associated rational knot, up to a sign and power of $u$. Recalling that $t=-u^{-2}$, we arrive at the following corollary, which seems to be new.
\end{remark}

\begin{corollary}\label{Jonesrationalform}
 The Jones polynomial of a rational knot $K$ satisfies \[V_K(t)V_K(t^{-1})=(2+t+t^{-1})-(1+t+t^{-1})a(t)a(t^{-1})\] where $a(t)\in\Z[t,t^{-1}]$.
\end{corollary}

\noindent Denote by $\gpC_0$ those elements of $\gpC$ for which $\alpha=0$.
\begin{lemma}  $\gpC_0=\{\pm\bfB_0\bfB_n\bfB_0|n\in\Z\}$.
\end{lemma}
\begin{proof}
In the case of $\alpha=0$, the condition in (\ref{eq:Cform}) reduces to $\beta\beta'=1$. Since $\beta\in\Z[u,u^{-1}]$, thus $\beta=\pm{}u^n$ for some $n\in\Z$. The matrix $\bfC$ is now determined by (\ref{eq:Cform}) to be
$\bfC=\begin{pmatrix}
0&i\beta\cr
i\beta'&\frac{\beta'-\beta}{u-u^{-1}}
\end{pmatrix}
=\pm\begin{pmatrix}
0&iu^n\cr
iu^{-n}&-[n]
\end{pmatrix}
=\mp\bfB_0\bfB_n\bfB_0$\end{proof}

\begin{definition}
    A pair of matrices $\bfC_1,\bfC_2\in\gpB$ are called a pivoting pair if they satisfy
    \[
   \begin{pmatrix}1&d\end{pmatrix}\bfC_1\begin{pmatrix}1\cr-d\end{pmatrix}
\cdot\begin{pmatrix}1&0\end{pmatrix}\bfC_2\begin{pmatrix}1\cr0\end{pmatrix}
=\begin{pmatrix}1&d\end{pmatrix}\bfC_2\begin{pmatrix}1\cr-d\end{pmatrix}
\cdot\begin{pmatrix}1&0\end{pmatrix}\bfC_1\begin{pmatrix}1\cr0\end{pmatrix}
    \]
    Equivalently, by (\ref{eq:Cprop}), $\alpha_1'\alpha_2=\alpha_1\alpha_2'$.
\end{definition}

\medskip\noindent It follows immediately from the definition that,
\begin{itemize}
    \item[(i)] for any $\bfC_0\in\gpC_0$ and $\bfC\in\gpB$, the pair $\bfC_0,\bfC$ is a pivoting pair;
    \item[(ii)] for pairs of matrices in $\gpB\backslash\gpC_0$, being a pivoting pair is an equivalence relation, namely it states that $\frac{\alpha_1}{\alpha_1'}=\frac{\alpha_2}{\alpha_2'}$. We denote the equivalence relation on the corresponding integer sequences by $\sim$.
\end{itemize}

\begin{example}\label{pivotingpairsbn} The matrices $\bfB_n\in\gpB$ all lie in the same equivalence class since $\alpha=[n]=\alpha'$. Similarly, any product of $\bfB_n$ matrices corresponding to a palindromic integer sequence will lie in this same equivalence class.
\end{example}

\begin{example}\label{generatingpivotingpairs}
For any $\bfC\in\gpB$, $n\in\Z$, the pair $\bfC,\bfC\bfB_n\bfC$ is a pivoting pair. This follows from the fact that the $(1,1)$ matrix entry of $\bfC_2\equiv\bfC\bfB_n\bfC$ is
\[
\alpha_2=\alpha\left(\frac{u^n\alpha-u^{-n}\alpha'}{u-u^{-1}}-u^n\beta-u^{-n}\beta'\right)
\]
where the factor in brackets is invariant under $u\rightarrow{}u^{-1}$. Thus all equivalence classes are infinite.
\end{example}

\begin{theorem}[\bf Pivoting Pairs Template]\label{pivotingpairstemplate} If two rational knots have integer sequence presentations related by the move
\[
(\bfn_0^{\epsilon_0},d_1,\bfn_1^{\epsilon_1},\ldots,d_k,\bfn_k^{\epsilon_k})
\longrightarrow
(\bfn_k^{\epsilon_k},d_k,\bfn_{k-1}^{\epsilon_{k-1}},\ldots,d_1,\bfn_1^{\epsilon_0})
\]
for some $d_1,\ldots,d_k\in\Z$, $\epsilon_0,\ldots,\epsilon_k\in\{1,*\}$ and integer sequences $\bfn_0\sim\bfn_1\sim\ldots\sim\bfn_k$,
then their Jones polynomials coincide.
\end{theorem}
\begin{proof}
The proof is as for Theorem \ref{firsttemplate} using Lemma \ref{keylemma}(2), now with $\bfU_j=\bfV_j=\bfC_j^{\epsilon_j}$. The first condition is automatically satisfied since $\bfU_j=\bfV_j$. The second condition follows from the pivot condition along with Lemma \ref{Bproperties}(ii).
\end{proof}

\begin{remark}Theorem \ref{pivotingpairstemplate} is a generalization of the first template (Theorem \ref{firsttemplate}), the latter being the special case where $\bfn_j$ are chosen to be identical $\forall{}j$.
\end{remark}
\begin{example}
Some examples of pivoting pairs which do not come from the general processes of Examples \ref{pivotingpairsbn} or \ref{generatingpivotingpairs} are $(5,3,3,-1)\sim(2,6,2,-2)\sim(1,6,3,4)$ while $(5,-3,-6,-4)\sim(5,-2,-4,-4)$. These are checked by brute force, that is, computation of $\alpha$. For example in the case of the last pair,
\[\alpha_1=-u^{-1}[3][4](u^2-1+u^{-2})(u^8+3u^6+5u^4+5u^2+5+4u^{-2}+3u^{-4}+2u^{-6}+u^{-8})\]
\[\alpha_2=-u^{-1}[5](u^8+3u^6+5u^4+5u^2+5+4u^{-2}+3u^{-4}+2u^{-6}+u^{-8})\]
in which the differing factors are all invariant under $u\rightarrow{}u^{-1}$.
This last pair is responsible for the Jones rational coincidence of Example \ref{firstpivotcase} using the pivoting pairs template with $k=1$, $\epsilon_0=\epsilon_1=1$ and $d_1=0$.
\end{example}

\begin{conjecture}
All Jones rational coincidences can be obtained by using the moves in the templates of Theorems \ref{pivotingpairstemplate} and \ref{secondtemplate}.
\end{conjecture}

\noindent This is verified explicitly in \cite{R} for all Jones rational coincidences with determinant $<900$.

\section{Jones rational coincidences, determinant $\!\!<900$}
Here is a list of all Jones rational coincidences with determinant $<900$, discounting mirror images, in order of increasing determinant. A star indicates an amphicheiral knot and a hash indicates pairs whose HOMFLYPT polynomials also coincide (suspected mutants). Triple coincidences are listed as such. Since rational knots are alternating, the span of their Jones polynomial gives the crossing number of the knot \cite{Kauffman}; the crossing numbers of knots here range from 10 in the first coincidence listed, up to 32 for the pair $(\frac{841}{782},\frac{841}{434})$.

\bigskip\noindent $(\frac{49}{22},\frac{49}{36})$, $(\frac{81}{44},\frac{81}{62})$,
$(\frac{121}{56},\frac{121}{100})$, $(\frac{121}{32},\frac{121}{76})$,
$(\frac{135}{62},\frac{135}{-28})$, $(\frac{147}{64},\frac{147}{106})$,
$(\frac{153}{70},\frac{153}{-32})$,

\smallskip\noindent $(\frac{161}{60},\frac{161}{74})$,
$(\frac{169}{40},\frac{169}{118})$, $(\frac{169}{90},\frac{169}{142})$,
$(\frac{171}{74},\frac{171}{-40  })$, $(\frac{189}{44},\frac{189}{-82})$,
$(\frac{207}{146},\frac{207}{164})$, $(\frac{209}{74},\frac{209}{150})$,

\smallskip\noindent
$(\frac{225}{106},\frac{225}{196})$, $(\frac{225}{122},\frac{225}{158})$,
$(\frac{231}{86},\frac{231}{170})$, $(\frac{243}{136},\frac{243}{190})$,
$(\frac{245}{106},\frac{245}{176})$, $(\frac{245}{108},\frac{245}{-88})$,
$(\frac{253}{104},\frac{253}{196})$,

\smallskip\noindent $(\frac{255}{92},\frac{255}{-112})$,
$(\frac{259}{96},\frac{259}{180})$, $(\frac{261}{148},\frac{261}{184})$,
$(\frac{275}{128},\frac{275}{228})$, $(\frac{275}{102},\frac{275}{202})$,
$(\frac{279}{58},\frac{279}{-128})$, $(\frac{279}{196},\frac{279}{214})$,

\smallskip\noindent $(\frac{287}{88},\frac{287}{130})$, $(\frac{289}{86},\frac{289}{120})$,
$(\frac{289}{50},\frac{289}{186})$, $(\frac{289}{152},\frac{289}{254})$,
$(\frac{297}{136},\frac{297}{-62})$, $(\frac{297}{116},\frac{297}{134})^\#$,
$(\frac{301}{80},\frac{301}{136})$,

\smallskip\noindent$(\frac{315}{68},\frac{315}{122})^{\!\#}$,
$(\frac{319}{196},\frac{319}{225})$, $(\frac{329}{68},\frac{329}{-214})$,
$(\frac{329}{122},\frac{329}{136})$, $(\frac{333}{76},\frac{333}{130})^{\!\#}$,
$(\frac{343}{148},\frac{343}{246})$, $(\frac{343}{90},\frac{343}{104})^{\!\#}$,

\smallskip\noindent
$(\frac{351}{190},\frac{351}{226})$, $(\frac{351}{152},\frac{351}{-82})$,
$(\frac{351}{136},\frac{351}{154})^\#$, $(\frac{361}{58},\frac{361}{248})$,
$(\frac{361}{172},\frac{361}{324})$, $(\frac{361}{134},\frac{361}{210})$,
$(\frac{361}{94},\frac{361}{284})$,

\smallskip\noindent $(\frac{363}{166},\frac{363}{298})$,
$(\frac{363}{98},\frac{363}{230})$, $(\frac{369}{200},\frac{369}{254})$,
$(\frac{369}{86},\frac{369}{-160})$, $(\frac{377}{70},\frac{377}{278})^*$,
$(\frac{385}{212},\frac{385}{268})$, $(\frac{399}{122},\frac{399}{164})$,

\smallskip\noindent
$(\frac{405}{224},\frac{405}{314})$, $(\frac{423}{88},\frac{423}{-194})$,
$(\frac{427}{192},\frac{427}{-88})$, $(\frac{441}{106},\frac{441}{358})$,
$(\frac{441}{202},\frac{441}{-92})$, $(\frac{441}{130},\frac{441}{-164})$,
$(\frac{441}{250},\frac{441}{304})$,

\smallskip\noindent
$(\frac{441}{188},\frac{441}{314})$, $(\frac{441}{230},\frac{441}{398})$,
$(\frac{451}{96},\frac{451}{260})$, $(\frac{451}{162},\frac{451}{184})$,
$(\frac{455}{186},\frac{455}{361})$, $(\frac{459}{100},\frac{459}{-206})$,
$(\frac{459}{316},\frac{459}{352})$,

\smallskip\noindent
$(\frac{473}{140},\frac{473}{305})$, $(\frac{473}{108},\frac{473}{174})$,
$(\frac{477}{214},\frac{477}{-\!104})$, $(\frac{477}{176},\frac{477}{-142})$,
$(\frac{481}{70},\frac{481}{-226})$, $(\frac{493}{194},\frac{493}{364})$,
$(\frac{495}{184},\frac{495}{-146})$,

\smallskip\noindent
$(\frac{495}{338},\frac{495}{392})$, $(\frac{495}{-218},\frac{495}{112},\frac{495}{292})$,
$(\frac{505}{212},\frac{505}{192})^\#$, $(\frac{505}{222},\frac{505}{-182})^*$,
$(\frac{507}{118},\frac{507}{352})$, $(\frac{507}{272},\frac{507}{428})$,

\smallskip\noindent
$(\frac{513}{124},\frac{513}{238})$, $(\frac{513}{226},\frac{513}{-116})$, $(\frac{513}{278},\frac{513}{350})$,
$(\frac{517}{216},\frac{517}{404})$, $(\frac{517}{190},\frac{517}{212})$,
$(\frac{527}{336},\frac{527}{-98})$, $(\frac{529}{484},\frac{529}{254})$,

\smallskip\noindent
$(\frac{529}{392},\frac{529}{488})$, $(\frac{529}{344},\frac{529}{68})$, $(\frac{529}{116},\frac{529}{208})$,
$(\frac{529}{160},\frac{529}{298})$, $(\frac{531}{196},\frac{531}{-158})$,
$(\frac{531}{124},\frac{531}{-344})$, $(\frac{533}{138},\frac{533}{216})$,

\smallskip\noindent
$(\frac{539}{246},\frac{539}{-118})$,
$(\frac{539}{244},\frac{539}{398})$, $(\frac{539}{386},\frac{539}{232})$,
$(\frac{549}{128},\frac{549}{376})$, $(\frac{551}{326},\frac{551}{402})$,
$(\frac{553}{360},\frac{553}{-114})$, $(\frac{567}{260},\frac{567}{-118})$,

\smallskip\noindent
$(\frac{567}{-\!440},\frac{567}{314})$,
$(\frac{581}{172},\frac{581}{338})^{\!\#}$\!, $(\frac{585}{268},\frac{585}{-\!122})$,
$(\frac{589}{88},\frac{589}{274})$, $(\frac{595}{156},\frac{595}{326})^{\!\#}$\!,
$(\frac{603}{218},\frac{603}{272})^{\!\#}$\!, $(\frac{605}{274},\frac{605}{494})$,

\smallskip\noindent
$(\frac{605}{164},\frac{605}{384})$,
$(\frac{621}{136},\frac{621}{343})^\#$, $(\frac{621}{352},\frac{621}{424})$,
$(\frac{621}{190},\frac{621}{280})^\#$, $(\frac{623}{404},\frac{623}{-130})$,
$(\frac{625}{426},\frac{625}{76})$, $(\frac{625}{224},\frac{625}{274})$,

\smallskip\noindent
$(\frac{625}{574},\frac{625}{324})$,
$(\frac{625}{474},\frac{625}{524})$, $(\frac{637}{288},\frac{637}{470})$,
$(\frac{637}{456},\frac{637}{272})$, $(\frac{639}{140},\frac{639}{-286})^\#$,
$(\frac{639}{418},\frac{639}{346})$, $(\frac{639}{362},\frac{639}{416})$,

\smallskip\noindent
$(\frac{649}{-\!240},\frac{649}{398})$,
$(\frac{651}{-\!292},\frac{651}{452})$, $(\frac{651}{446},\frac{651}{-142})$,
$(\frac{651}{460},\frac{651}{502})$, $(\frac{657}{-\!290},\frac{657}{148})^{\!\#}$\!\!,
$(\frac{657}{356},\frac{657}{446})$, $(\frac{663}{370},\frac{663}{-\!98})$,

\smallskip\noindent
$(\frac{665}{142},\frac{665}{492})$,
$(\frac{675}{-584},\frac{675}{314})$, $(\frac{675}{-154},\frac{675}{296})^\#$,
$(\frac{675}{152},\frac{675}{-298})^\#$, $(\frac{679}{180},\frac{679}{-492})$,
$(\frac{679}{284},\frac{679}{402})$,

\smallskip\noindent
$(\frac{689}{314},\frac{689}{-110})$, $(\frac{693}{304},\frac{693}{502},\frac{693}{250})$,
$(\frac{693}{184},\frac{693}{382},\frac{693}{436})$, $(\frac{703}{146},\frac{703}{516})$,
$(\frac{703}{262},\frac{703}{-308})$, $(\frac{711}{245},\frac{711}{326})$,

\smallskip\noindent
$(\frac{711}{-166},\frac{711}{464})$, $(\frac{711}{500},\frac{711}{-392})$,
$(\frac{715}{552},\frac{715}{-292})$,
$(\frac{715}{412},\frac{715}{588})$, $(\frac{721}{324},\frac{721}{-500})$,
$(\frac{721}{-556},\frac{721}{-268})$,

\smallskip\noindent
$(\frac{721}{318},\frac{721}{332})$,
$(\frac{729}{296},\frac{729}{136})$, $(\frac{729}{676},\frac{729}{352})$,
$(\frac{729}{-170},\frac{729}{496})$,
$(\frac{729}{478},\frac{729}{-152})$, $(\frac{729}{458},\frac{729}{512})$,
$(\frac{729}{620},\frac{729}{188})$,

\smallskip\noindent
$(\frac{729}{406},\frac{729}{568})$,
$(\frac{735}{302},\frac{735}{-412})$, $(\frac{735}{314},\frac{735}{524})$,
$(\frac{737}{302},\frac{737}{570})$,
$(\frac{737}{332},\frac{737}{600})$, $(\frac{737}{126},\frac{737}{394})$,
$(\frac{741}{548},\frac{741}{158})$,

\smallskip\noindent
$(\frac{745}{288},\frac{745}{-462})$,
$(\frac{745}{328},\frac{745}{-442})$, $(\frac{747}{232},\frac{747}{-266})$,
$(\frac{747}{-526},\frac{747}{470})$,$(\frac{749}{406},\frac{749}{-550})$, $(\frac{749}{-326},\frac{749}{530})$,

\smallskip\noindent
$(\frac{755}{312},\frac{755}{-468})$,
$(\frac{755}{272},\frac{755}{448})$,
$(\frac{759}{448},\frac{759}{-172})$, $(\frac{765}{362},\frac{765}{668})$,
$(\frac{765}{592},\frac{765}{-428})$,
$(\frac{777}{458},\frac{777}{214})^{\!\#}$\!, $(\frac{779}{162},\frac{779}{572})$,

\smallskip\noindent
$(\frac{783}{350},\!\frac{783}{-\!172})$, $(\frac{783}{620},\!\frac{783}{-\!424})$,
$(\frac{783}{338},\frac{783}{164})$, $(\frac{783}{-\!568},\frac{783}{476})$,
$(\frac{791}{466},\frac{791}{234})^{\!\#}$\!\!,
$(\frac{793}{342},\frac{793}{-\!146})$, $(\frac{801}{-\!578},\frac{801}{590})$,

\smallskip\noindent
$(\frac{801}{454},\frac{801}{544})$, $(\frac{801}{634},\frac{801}{-434})$,
$(\frac{803}{288},\frac{803}{310})^{\!\#}$\!\!, $(\frac{817}{124},\frac{817}{382})$,
$(\frac{819}{590},\!\frac{819}{-\!346})$,
$(\frac{819}{184},\frac{819}{-\!362})$, $(\frac{819}{374},\frac{819}{-562})$,

\smallskip\noindent
$(\frac{819}{628},\!\frac{819}{-\!464})$, $(\frac{833}{-\!342},\frac{833}{484})$,
$(\frac{833}{540},\!\frac{833}{-\!174})$, $(\frac{833}{232},\frac{833}{218})^{\!\#}$\!,
$(\frac{833}{-\!536},\frac{833}{596})$,
$(\frac{837}{604},\frac{837}{512})$, $(\frac{837}{344},\frac{837}{158})$,

\smallskip\noindent
$(\frac{837}{260},\frac{837}{-298})$, $(\frac{841}{726},\frac{841}{204})$,
$(\frac{841}{378},\frac{841}{-260})$, $(\frac{841}{550},\frac{841}{86})$,
$(\frac{841}{608},\frac{841}{318})$,
$(\frac{841}{666},\frac{841}{144})$, $(\frac{841}{782},\frac{841}{434})$,

\smallskip\noindent
$(\frac{845}{-\!586},\frac{845}{194})$, $(\frac{845}{454},\frac{845}{714})$,
$(\frac{847}{692},\!\frac{847}{-\!386})$, $(\frac{847}{354},\!\frac{847}{-\!614})$,
$(\frac{847}{540},\frac{847}{232})$,
$(\frac{855}{-\!178},\frac{855}{392})$, $(\frac{867}{358},\!\frac{867}{-\!254})$,

\smallskip\noindent
$(\frac{867}{188},\!\frac{867}{-\!526})$, $(\frac{867}{562},\!\frac{867}{-\!154})$,
$(\frac{867}{460},\frac{867}{766})$, $(\frac{871}{-\!720},\frac{871}{138})$,
$(\frac{873}{-\!182},\frac{873}{400})$,
$(\frac{873}{272},\!\frac{873}{-\!310})$, $(\frac{875}{256},\frac{875}{506})$,

\smallskip\noindent
$(\frac{889}{248},\frac{889}{502})$, $(\frac{891}{-208},\frac{891}{386})$,
$(\frac{891}{692},\frac{891}{494})$, $(\frac{893}{576},\frac{893}{-184})$

\medskip\noindent It is striking to see the prevalence of squares or at least multiplicities in the prime decomposition of the determinant in Jones rational coincidences. However a direct condition on a pair $(\frac{p}{q},\frac{p}{q'})$ to form a Jones rational coincidence has yet to be found (other than direct computation of the associated Jones polynomials!)

\end{document}